\documentclass[11pt,oneside]{amsart}
\usepackage[latin9]{inputenc}
\usepackage{geometry}
\usepackage{color}
\usepackage{mathrsfs}
\usepackage{amstext}
\usepackage{amsthm}
\usepackage{amssymb}
\usepackage{graphicx}
\usepackage{setspace}
\PassOptionsToPackage{normalem}{ulem}
\usepackage{ulem}
\usepackage[unicode=true,pdfusetitle,
 bookmarks=true,bookmarksnumbered=false,bookmarksopen=false,
 breaklinks=false,pdfborder={0 0 1},backref=page,colorlinks=true]
 {hyperref}
\hypersetup{
 linkcolor=blue}

\makeatletter
\numberwithin{equation}{section}
\numberwithin{figure}{section}
\theoremstyle{plain}
\newtheorem{thm}{\protect\theoremname}[section]
\theoremstyle{remark}
\newtheorem{rem}[thm]{\protect\remarkname}
\theoremstyle{plain}
\newtheorem{cor}[thm]{\protect\corollaryname}
\theoremstyle{definition}
\newtheorem{defn}[thm]{\protect\definitionname}
\theoremstyle{plain}
\newtheorem{lem}[thm]{\protect\lemmaname}
\theoremstyle{plain}
\newtheorem{prop}[thm]{\protect\propositionname}

\usepackage{graphicx}
\usepackage{appendix}

\usepackage{enumitem}
\usepackage[backref=page]{hyperref}

\makeatother

\providecommand{\corollaryname}{Corollary}
\providecommand{\definitionname}{Definition}
\providecommand{\lemmaname}{Lemma}
\providecommand{\propositionname}{Proposition}
\providecommand{\remarkname}{Remark}
\providecommand{\theoremname}{Theorem}

\begin{document}
\global\long\def\F{\mathrm{\mathbf{F}} }%
\global\long\def\Aut{\mathrm{Aut}}%
\global\long\def\C{\mathbf{C}}%
\global\long\def\H{\mathcal{H}}%
\global\long\def\U{\mathcal{U}}%
\global\long\def\P{\mathcal{P}}%
\global\long\def\ext{\mathrm{ext}}%
\global\long\def\hull{\mathrm{hull}}%
\global\long\def\triv{\mathrm{triv}}%
\global\long\def\Hom{\mathrm{Hom}}%

\global\long\def\trace{\mathrm{tr}}%
\global\long\def\End{\mathrm{End}}%

\global\long\def\L{\mathcal{L}}%
\global\long\def\W{\mathcal{W}}%
\global\long\def\E{\mathbb{E}}%
\global\long\def\SL{\mathrm{SL}}%
\global\long\def\R{\mathbf{R}}%
\global\long\def\Pairs{\mathrm{PowerPairs}}%
\global\long\def\Z{\mathbf{Z}}%
\global\long\def\rs{\to}%
\global\long\def\A{\mathcal{A}}%
\global\long\def\a{\mathbf{a}}%
\global\long\def\rsa{\rightsquigarrow}%
\global\long\def\D{\mathbf{D}}%
\global\long\def\b{\mathbf{b}}%
\global\long\def\df{\mathrm{def}}%
\global\long\def\eqdf{\stackrel{\df}{=}}%
\global\long\def\ZZ{\overline{Z}}%
\global\long\def\Tr{\mathrm{Tr}}%
\global\long\def\N{\mathbf{N}}%
\global\long\def\std{\mathrm{std}}%
\global\long\def\HS{\mathrm{H.S.}}%
\global\long\def\e{\mathbf{e}}%
\global\long\def\c{\mathbf{c}}%
\global\long\def\d{\mathbf{d}}%
\global\long\def\AA{\mathbf{A}}%
\global\long\def\BB{\mathbf{B}}%
\global\long\def\u{\mathbf{u}}%
\global\long\def\v{\mathbf{v}}%
\global\long\def\spec{\mathrm{spec}}%
\global\long\def\Ind{\mathrm{Ind}}%
\global\long\def\half{\frac{1}{2}}%
\global\long\def\Re{\mathrm{Re}}%
\global\long\def\Im{\mathrm{Im}}%
\global\long\def\Rect{\mathrm{Rect}}%
\global\long\def\Crit{\mathrm{Crit}}%
\global\long\def\Stab{\mathrm{Stab}}%
\global\long\def\SL{\mathrm{SL}}%

\title{Explicit spectral gaps for random covers of Riemann surfaces}
\author{Michael Magee and Frédéric Naud}
\begin{abstract}
We introduce a permutation model for random degree $n$ covers $X_{n}$
of a non-elementary convex-cocompact hyperbolic surface $X=\Gamma\backslash\mathbb{H}$.
Let $\delta$ be the Hausdorff dimension of the limit set of $\Gamma$.
We say that a resonance of $X_{n}$ is \emph{new }if it is not a resonance
of $X$, and similarly define new eigenvalues of the Laplacian.

We prove that for any $\epsilon>0$ and $H>0$, with probability tending
to $1$ as $n\to\infty$, there are no new resonances $s=\sigma+it$
of $X_{n}$ with $\sigma\in[\frac{3}{4}\delta+\epsilon,\delta]$ and
$t\in[-H,H]$. This implies in the case of $\delta>\frac{1}{2}$ that
there is an explicit interval where there are no new eigenvalues of
the Laplacian on $X_{n}$. By combining these results with a deterministic
`high frequency' resonance-free strip result, we obtain the corollary
that there is an $\eta=\eta(X)$ such that with probability $\to1$
as $n\to\infty$, there are no new resonances of $X_{n}$ in the region
$\{\,s\,:\,\mathrm{Re}(s)>\delta-\eta\,\}$.
\end{abstract}

\maketitle
\tableofcontents{}

\section{Introduction}

This paper is about spectral gaps for random Riemann surfaces. More
specifically, we are interested in various notions of spectral gap
for random covers of a fixed Schottky Riemann surface. This is in
close analogy to questions about the spectral gap of a random regular
graph, and this analogy informs our model for random coverings, so
we begin with a discussion on graphs.

Let $G$ be a $k$-regular graph on $n$ vertices. Then the adjacency
matrix $A_{G}$ of $G$ has $n$ real eigenvalues in $[-k,k]$ and
$k$ appears as an eigenvalue with multiplicity equal to the number
of connected components of $G$. Denoting by
\[
k=\lambda_{0}\geq\lambda_{1}\geq\cdots\geq\lambda_{n}
\]
 the eigenvalues of $G$, the \emph{spectral gap} of $G$ is $\lambda_{0}-\lambda_{1}$.
If $G$ is connected, then $\lambda_{0}>\lambda_{1}$ and the spectral
gap is related to the exponential rate at which the random walk on
$G$ converges to the uniform measure. As such, it is an important
quantity in theoretical computer science, and accordingly, there has
been a great deal of interest in the spectral gap of a \emph{random}
regular graph. Alon's conjecture \cite{Alon}, now a theorem due to
Friedman \cite{Friedman}, says that for any $\epsilon>0$, as $n\to\infty$,
the probability that $\lambda_{1}(G_{n})>2\sqrt{k-1}+\epsilon$ tends
to zero, when $G_{n}$ is sampled uniformly at random from $k$-regular
graphs with $n$ vertices. The relevance of the quantity $2\sqrt{k-1}$
is that for any $k$-regular graph with $n$ vertices, a result of
Alon-Boppana \cite{Nilli} says that $\lambda_{1}(G)\geq2\sqrt{k-1}-o_{n}(1)$,
so $2\sqrt{k-1}$ is an asymptotically optimal lower bound for $\lambda_{1}(G)$,
often called the \emph{Ramanujan bound} after \cite{LPS}.

The model of a random graph described above chooses random graphs
according to a uniform distribution. Another popular model for a random
$2k$-regular graph is called the \emph{permutation model }and is
the one we wish to focus on in the sequel. Let $\Gamma=\langle\gamma_{1},\ldots,\gamma_{k},\gamma_{1}^{-1},\ldots,\gamma_{k}^{-1}\rangle$
be a free group on $k$ generators, $k\geq2$, and let $S_{n}$ denote
the symmetric group on $n$ letters, and $\phi_{n}$ be a random homomorphism
from $\Gamma$ to $S_{n}$, sampled uniformly from all possible homomorphisms.
Since $\Gamma$ is free, a homomorphism is described simply by choosing
the images $\phi_{n}(\gamma_{i})$ of the generators of $\Gamma$
independently and uniformly from $S_{n}$. Then let $G_{n}$ be the
random graph with vertex set $[n]\eqdf\{1,\ldots,n\}$ and with an
edge between $i$ and $j$ if there is a generator $\gamma_{a}$ such
that $\phi_{n}(\gamma_{a})(i)=j$. We will adapt this model to a model
of a random Riemann surface. 

Let $X$ be a connected, non-elementary, non-compact, convex co-compact
hyperbolic surface. Then $X=\Gamma\backslash\mathbb{H}$ where $\mathbb{H}$
is the hyperbolic upper half plane and $\Gamma$ is a free subgroup
of $\SL_{2}(\R)$. We view $X$ as fixed throughout the paper. We
let $X_{n}$ be the random $n$-cover of $X$ obtained as a fibered
product $X_{n}\eqdf\mathbb{H}\times_{\phi_{n}}[n]$. More precisely,
$X_{n}$ is the quotient of $\mathbb{H}\times[n]$ by the diagonal
action of $\Gamma$
\[
\gamma.(x,i)=(\gamma(x),\phi_{n}(\gamma)(i)).
\]
If $S\subset[n]$ is a set of representatives for the orbits of $\Gamma$
on $[n]$ via $\phi_{n}$, and $\Gamma_{i}\eqdf\Stab_{\Gamma}(i)$
is the stabilizer of $i\in S$, then $X_{n}$ is isomorphic to the
disjoint union of (connected) covers $\Gamma_{i}\backslash\mathbb{H}$,
i.e. 
\[
X_{n}=\bigsqcup_{i\in S}\Gamma_{i}\backslash\mathbb{H}.
\]
Notice that we have 
\[
\sum_{i\in S}[\Gamma:\Gamma_{i}]=n.
\]
We say that a property $E(\phi_{n})$ of the random $\phi_{n}$ holds
\emph{asymptotically almost surely} (a.a.s.) if as $n\to\infty$,
the probability that $E(\phi_{n})$ holds tends to $1$. It is an
elementary calculation\footnote{For a concrete reference, this statement follows from \cite[Thm.13]{BroderShamir}.
One can also prove by elementary combinatorial arguments that the
probability that $X_{n}$ is connected as $n\rightarrow+\infty$ is
bigger than $1-\frac{C(k)}{n^{k-1}}$, where $C(k)>0$ is a constant
depending only on $k$.} that a.a.s. $\Gamma$ acts transitively on $[n]$ via $\phi_{n}$
and hence, a.a.s. $X_{n}$ is connected. This also follows from the
main theorems below. Although we do not assume $X_{n}$ is connected
at any point, it would not hurt to assume this on a first reading.

We now discuss the spectral theory of $X$ and $X_{n}$. The group
$\Gamma$ acts properly discontinuously on $\mathbb{H},$ but for
any point $o\in\mathbb{H}$, the orbit $\Gamma o$ accumulates on
$\partial\mathbb{H}=\R\cup\{\infty\}$ and the accumulation set of
this orbit is called the \emph{limit set of $\Gamma$ }and denoted
by $\Lambda(\Gamma)$. This $\Lambda(\Gamma)$ is a perfect nowhere
dense fractal and has an associated Hausdorff dimension $\delta\eqdf\dim_{\mathrm{Haus}}(\Lambda(\Gamma))\in[0,1)$.
By a result of Lax and Phillips \cite{LP}, the spectrum of the Laplacian
$\Delta_{X}$ is discrete in the range $[0,\frac{1}{4})$, and Patterson
\cite{Patterson} proved that if $\delta>\frac{1}{2}$, then the lowest
eigenvalue of $\Delta_{X}$ is $\delta(1-\delta)$. If $\delta\leq\frac{1}{2}$
then there are no eigenvalues of $\Delta_{X}$. The same is true for
$X_{n}$, with the same $\delta$ (although $\delta(1-\delta)$ will
be simple if and only if $X_{n}$ is connected). More generally, if
$\lambda$ is any eigenvalue of $X$, then by lifting eigenfunctions
through the covering map, $\lambda$ is an eigenvalue for $X_{n}$
with at least as large multiplicity. 

The first main theorem of our paper is the following.
\begin{thm}
\label{thm:L2-gap}Assume that $\delta>\frac{1}{2}$. Then for any
$\sigma_{0}\in\left(\frac{3}{4}\delta,\delta\right)$, a.a.s.
\begin{equation}
\spec(\Delta_{X_{n}})\cap\left[\delta\left(1-\delta\right),\sigma_{0}(1-\sigma_{0})\right]=\spec(\Delta_{X})\cap\left[\delta\left(1-\delta\right),\sigma_{0}(1-\sigma_{0})\right]\label{eq:gap-1}
\end{equation}
and the multiplicities on both sides are the same.
\end{thm}

\begin{rem}
This theorem implies that a.a.s. the $X_{n}$ have a uniform spectral
gap, and this spectral gap only depends on $\delta$ and the gap between
the first two eigenvalues of $X$.
\end{rem}

\begin{rem}
If $\delta\in(\frac{1}{2},\frac{2}{3})$ then since $X_{n}$ has no
eigenvalues in $[\frac{1}{4},\infty)$ by a result of Lax and Phillips
\cite{LP}, Theorem \ref{thm:L2-gap} implies that a.a.s. $X_{n}$
has no new eigenvalues. 
\end{rem}

\begin{rem}
Theorem \ref{thm:L2-gap} can be viewed as a significant sharpening
of a result of Brooks and Makover \cite{BrooksMakover}, albeit in
the infinite area setting. See $\S$\ref{subsec:Prior-work} for a
more detailed discussion of this comparison.
\end{rem}

\begin{rem}
We point out that it is possible for $X_{n}$ to not be connected,
and in this case, there is no spectral gap. Even further, it is easy
to see that $X_{n}$ can be a connected cyclic cover of $X$, and
by results of \cite{JNS}, these have no uniform spectral gap.
\end{rem}

\begin{rem}
In the limit as $\delta\to1$, the range of forbidden eigenvalues
in (\ref{eq:gap-1}) becomes $[0,\frac{3}{16})$. This is interestingly
the same range covered by Selberg's $\frac{3}{16}$ Theorem \cite{Selberg}
on the spectral gap of congruence covers of the modular surface $\SL_{2}(\Z)\backslash\mathbb{H}.$
This should also be compared to the deterministic result of Gamburd
\cite{Gamburd1} for congruence covers of infinite index geometrically
finite subgroups of $\SL_{2}(\Z)$: assuming $\delta>\frac{5}{6}$,
he shows that the spectrum remains the same in the range $[\delta(1-\delta),\frac{5}{36})$.
See also \cite{Magee} for a generalization of this result to higher
dimensions.
\end{rem}

We write $\chi(X)$ for the Euler characteristic of $X$. It has recently
been proved by Ballmann, Matthiesen, and Mondal \cite{BMM} that if
$\chi(X)<0$, $\Delta_{X}$ has at most $-\chi(X)$ eigenvalues. If
$\chi(X)=-1$ then this means the only possible eigenvalue of $X$
is at $\delta(1-\delta)$ and thus Theorem \ref{eq:gap-1} yields
\begin{cor}
Assume that $\delta>\frac{1}{2}$. If $X$ is topologically a pair
of pants, or a torus with one hole, then for any $\sigma_{0}\in\left(\frac{3}{4}\delta,\delta\right)$,
a.a.s.
\[
\spec(\Delta_{X_{n}})\cap(\delta\left(1-\delta\right),\sigma_{0}(1-\sigma_{0})]=\emptyset,
\]
and $\delta(1-\delta)$ is a simple eigenvalue of $\Delta_{X_{n}}$.
\end{cor}

We now turn to what we can say about general $\delta\in(0,1)$. In
the case $\delta\leq\half$, $\Delta_{X}$ and $\Delta_{X_{n}}$ will
have no discrete $L^{2}$ spectrum, so one must consider a more subtle
notion of spectral gap. 

For any non-elementary convex co-compact hyperbolic $Y$ with $\delta=\delta(Y)$
(e.g. $Y=X$, $Y=X_{n}$) the \emph{resolvent}
\[
R_{Y}(s)\eqdf(\Delta_{Y}-s(1-s))^{-1}:L^{2}(Y)\rightarrow L^{2}(Y)
\]
is, a priori, a meromorphic family of bounded operators in the right
half plane $\mathrm{Re}(s)>\half$ with poles precisely at real $s$
such that $s(1-s)$ is an eigenvalue of $\Delta_{X}.$ By work of
Mazzeo-Melrose \cite{MazzeoMelrose}, it can be meromorphically continued
to a family of bounded operators from $C_{0}^{\infty}(Y)\rightarrow C^{\infty}(Y)$
that is meromorphic in $s\in\C$. In the case of hyperbolic surfaces,
a simpler proof of the meromorphic continuation is due to Guillopé
and Zworski \cite{GZ1}, see also the book \cite{Borthwick}.

The poles of the meromorphically continued resolvent are called \textit{resonances}\emph{
of $Y$}. In the sequel we write $\mathcal{R}_{Y}\subset\C$ for the
multi-set of resonances, repeated according to multiplicities\footnote{Following \cite[Def. 8.2]{Borthwick}, the multiplicity of a resonance
$s$ of $Y$ is given by $\mathrm{rank}\left(\int_{\gamma}R_{Y}(s)ds\right)$
where $\gamma$ is an anticlockwise oriented circle enclosing $s$
and no other resonance of $Y$. }. Resonances, unlike $L^{2}$-eigenvalues, correspond to a non self-adjoint
spectral problem and are therefore notoriously difficult to study.
There is however a clear analog of the spectral gap in this setting.
The `bass resonance' is located at $s=\delta$ and by a result of
Naud \cite{NaudExpanding} if $Y$ is connected then there exists
a constant $\epsilon_{\Gamma}>0$ such that 
\[
\mathcal{R}_{Y}\cap\{\,s\,:\,\mathrm{Re}(s)\geq\delta-\epsilon_{\Gamma}\,\}=\{\delta\}.
\]
We call the existence of such a resonance free strip a \emph{spectral
gap} for $Y$. The spectral gap on hyperbolic surfaces has numerous
applications, from prime geodesic theorems \cite{Naudasymptotics}
to local $L^{2}$-asymptotics of waves \cite{GN1}. A recent breakthrough
of Bourgain-Dyatlov \cite{BDgap} showed that there always exists
an \textit{``essential spectral gap}'' past the line $\{\Re(s)=\frac{1}{2}\}$,
i.e. there exists $\tilde{\epsilon}=\tilde{\epsilon}(Y)>0$ such that
\[
\mathcal{R}_{Y}\cap\{\,s\,:\,\mathrm{Re}(s)\geq\frac{1}{2}-\tilde{\epsilon}\,\}
\]
is a finite set. The proof is based on the general phenomenon of ``fractal
uncertainty principle'', see \cite{DyatlovFUP}. We point out that
$\widetilde{\epsilon}>0$ can be made explicit, see Jin-Zhang \cite{JZ1}
and also Dyatlov-Jin \cite{DJ1}. For a broader view and a state of
the art survey on the mathematical theory of resonances including
hyperbolic manifolds and related conjectures, we recommend to read
\cite{Zworski_survey}. Our next main result is the following.
\begin{thm}
\noindent \label{thm:rectangle-gap}Fix any $H>0$ and $\sigma_{0}\in\left(\frac{3}{4}\delta,\delta\right)$,
and let
\[
\Rect(\sigma_{0},H)\eqdf\{\,s=\sigma+it\,:\,\sigma\in[\sigma_{0},\delta]\ \mathrm{and}\ \vert t\vert\leq H\,\}.
\]
Then a.a.s.
\[
\mathcal{R}_{X_{n}}\bigcap\Rect(\sigma_{0},H)=\mathcal{R}_{X}\bigcap\Rect(\sigma_{0},H)
\]
where the multiplicities on both sides are the same.
\end{thm}

\begin{rem}
Because all eigenvalues $\lambda_{\sigma}=\sigma(1-\sigma)$ of $\Delta_{X_{n}}$
with $\sigma>\frac{1}{2}$ give a resonance of $X_{n}$ at $\sigma$,
with the same multiplicity, and the same is true for $X$, Theorem
\ref{thm:rectangle-gap} implies Theorem \ref{thm:L2-gap} and extends
it to resonances in rectangles of explicit width and any bounded height
$2H$. We point out that Theorem \ref{thm:rectangle-gap} actually
yields new information on low frequency resonances past the line $\{\Re(s)=\frac{1}{2}\}$
when $\delta\in(\frac{1}{2},\frac{2}{3})$.
\end{rem}

This leaves the question of how to deal with resonances with large
imaginary part. For this we have the following theorem that applies
to arbitrary covers. Note that here there is no randomness involved.
\begin{thm}
\label{thm:high-frequency} Assume that $\Gamma$ is a non-elementary
convex co-compact group. Then there exist $\epsilon_{\Gamma}>0$ and
$T_{\Gamma}>0$ such that for all finite index subgroups $\widetilde{\Gamma}\subset\Gamma$,
we have for $\widetilde{X}=\widetilde{\Gamma}\backslash\mathbb{H}$,
\[
\mathcal{R}_{\widetilde{X}}\cap\{\,s\,:\,\Re(s)\geq\delta-\epsilon_{\Gamma}\ \mathrm{and}\ \vert\Im(s)\vert\geq T_{\Gamma}\,\}=\emptyset.
\]
\end{thm}

\begin{rem}
\textit{\emph{F}}rom the work of Bourgain and Dyatlov \cite{BDFourier},
we know that there exists $\varepsilon(\delta)>0$, depending only
on $\delta$ and thus uniform on covers such that 
\[
\mathcal{R}_{X}\cap\{\Re(s)\geq\delta-\varepsilon(\delta)\}
\]
is a \textbf{finite} set. However the result of Bourgain and Dyatlov
does not provide any information on the finite set of resonances in
this uniform strip. \textit{\emph{Theorem \ref{thm:high-frequency}
}}shows that new resonances can only appear in a compact region. 
\end{rem}

\noindent \bigskip{}

\noindent Combining Theorem \ref{thm:rectangle-gap} with Theorem
\ref{thm:high-frequency} yields the following corollary. 
\begin{cor}
\noindent \label{cor:A.a.s.-the-random-covers-have-spectral-gap}A.a.s.
the random cover $X_{n}\rightarrow X$ has a uniform spectral gap.
In particular, above each non elementary surface $X$, one can produce
an infinite family of covers $X_{n}$ with degree $n$ and having
a uniform spectral gap.
\end{cor}

\begin{rem}
\label{rem:coro-remark}When $\delta>\frac{1}{2}$, Corollary \ref{cor:A.a.s.-the-random-covers-have-spectral-gap}
follows from a mild extension of \cite[Thm. 1.2]{BGS2} together with
results on random graphs as explained in $\S$\ref{subsec:Prior-work}.
However, when $\delta\leq\frac{1}{2}$, to our knowledge, Corollary
\ref{cor:A.a.s.-the-random-covers-have-spectral-gap} is completely
new: the only result of that type so far is for congruence covers
of convex co-compact subgroups of $\mathrm{SL}_{2}(\Z)$, see Oh-Winter
\cite{OW} and the discussion below.
\end{rem}

\subsection{Prior work\label{subsec:Prior-work}}

\textbf{Brooks and Makover. }Brooks and Makover in \cite{BrooksMakover}
consider a similar model for random finite area Riemann surfaces.
In this model, random surfaces are modeled by random 3-regular oriented
graphs sampled according to a refinement of the Bollobás `bin model'
introduced in \cite{Bollobas}. Then Brooks and Makover \cite{BrooksMakover}
construct from a random oriented graph on $n$ vertices a Riemann
surface $Y_{n}$, tiled by a specific hyperbolic triangle with one
vertex at $\infty$. They then consider a compactification $Y_{n}^{c}$
of the cusped surface $Y_{n}$. Thus $Y_{n}^{c}$ is a random compact
Riemann surface; the genus of $Y_{n}^{c}$ is however not deterministic\footnote{By a result of Gamburd \cite{GamburdBelyi}, if $l(Y_{n})\eqdf\frac{n}{2}+2-2\mathrm{genus}(Y_{n}^{c})$,
then as $n\to\infty$, $l(Y_{n})$ converges to a Poisson-Dirichlet
distribution. The function $l(Y_{n})$ coincides with the number of
cusps of $Y_{n}$.}. Brooks and Makover proved in \emph{(ibid.)}
\begin{thm}[Brooks-Makover]
\label{thm:Brooks-Makover}There is some constant $C>0$ such that
a.a.s. the first non-zero eigenvalue of $Y_{n}^{c}$ is $\geq C$.
\end{thm}

Although our main theorems deal instead with infinite area Riemann
surfaces, they offer two improvements over Theorem \ref{thm:Brooks-Makover}:
\begin{itemize}
\item The range of new forbidden eigenvalues and resonances in Theorems
\ref{thm:L2-gap} and \ref{thm:rectangle-gap} are explicit,
\item Moreover, we have an entire moduli space of random families (parameterized
by the modulus of $X$) and the range of forbidden eigenvalues and
resonances only depends on $X$ in a very mild way, through the Hausdorff
dimension of the limit set.
\end{itemize}
\textbf{The Brooks-Burger transfer principle.}\textbf{\emph{ }}Also
relevant to the current work is the following transfer principle for
small eigenvalues developed independently by Brooks and Burger in
\cite{Brooks,Burger}.
\begin{thm}[Brooks-Burger]
\label{thm:brooks-burger}Let $Y$ be any compact Riemannian manifold
with $\Gamma=\pi_{1}(Y)$. There is a constant $c(Y)>0$ and a finite
subset $S\subset\Gamma$ such that the following hold. Let $\Gamma'$
be any finite index subgroup of $\Gamma$, with associated Riemannian
covering space $Y'$ of $Y$. Let $\lambda_{1}(Y')$ be such that
$\spec(\Delta_{Y'})=\{\,0\leq\lambda_{1}\leq\lambda_{2}\leq\ldots\,\}.$
Let $G=G(\Gamma',S)$ be the Schreier coset graph of $S$ acting on
$\Gamma/\Gamma'$. Then
\begin{equation}
\lambda_{1}(Y')\geq c(Y)\left(\lambda_{0}\left(G\right)-\lambda_{1}(G)\right).\label{eq:eigenvalue-inequality}
\end{equation}
\end{thm}

Theorem \ref{thm:brooks-burger} was extended to Galois covers of
non-elementary convex co-compact hyperbolic surfaces by Bourgain,
Gamburd and Sarnak in \cite[Thm. 1.2]{BGS2} where the left hand side
of (\ref{eq:eigenvalue-inequality}) is replaced by the gap between
$\delta(1-\delta)$ and the next eigenvalue of the $L^{2}$-Laplacian.
This extends to non-Galois covers and therefore applies in the setting
of this paper as follows.

Let us assume that $X_{n}=\Gamma_{n}\backslash\mathbb{H}$ is connected,
for simplicity, although the argument can be adapted to the general
case. For fixed $S\subset\Gamma$, the Schreier coset graphs $G_{n}$
of $S$ acting on $\Gamma/\Gamma_{n}\cong[n]$ are precisely the random
regular graphs of the permutation model, and a.a.s. these have a uniform
spectral gap by \cite{BroderShamir,Friedman}. Hence by the extension
of \cite[Thm. 1.2]{BGS2} the $X_{n}$ have a uniform spectral gap
between $\delta(1-\delta)$ and the next $L^{2}$-eigenvalue. Importantly,
in all versions of Theorem \ref{thm:brooks-burger}, the constant
$c$ depends on $Y$ in a complicated way. Because of this, it is
unlikely such an argument would lead to e.g. Theorem \ref{thm:L2-gap}.
However, this argument does lead to Corollary \ref{cor:A.a.s.-the-random-covers-have-spectral-gap}
when $\delta>\frac{1}{2}$ (cf. Remark \ref{rem:coro-remark}).

It is also worth mentioning that a variant of Theorem \ref{thm:brooks-burger}
has also been developed for resonances in \cite{BGS2,OW,MOW}, for
specific congruence coverings of $Y=\Gamma\backslash\mathbb{H}$ where
$\Gamma$ is an infinite index subgroup of $\SL_{2}(\Z)$. Besides
only dealing with Galois covers, the key reason that these methods
cannot prove Corollary \ref{cor:A.a.s.-the-random-covers-have-spectral-gap}
when $\delta\leq\frac{1}{2}$ is the following. The state of the art
method \cite[Appendix]{MOW} of dealing with low frequency resonances
(a la Theorem \ref{thm:rectangle-gap}) involves bounds on the dimensions
of non-trivial irreducible representations of finite groups $\mathcal{G}$
that are polynomial in $|\mathcal{G}|$. The relevant groups in our
setting are $S_{n}$, and the issue is that $S_{n}$ has non-trivial
irreducible representations of dimensions that are sub-logarithmic
in $|S_{n}|=n!$.

Finally we point out that the methods of \cite{BGS2,OW,MOW} are not
well adapted to efficiently tracking constants and hence likely not
suitable for producing explicit resonance free regions as in Theorem
\ref{thm:rectangle-gap}.

\subsection{Overview of proofs and paper organization\label{subsec:An-overview-of}}

All the proofs of the paper rely on a Schottky encoding of the action
of $\Gamma$ on $\R$ that is presented in $\S$\ref{subsec:Words,-encodings-of}.
To control resonances (and eigenvalues) we rely on the connection
between resonances and zeros of the Selberg zeta function due to Patterson
and Perry \cite{PattersonPerry}. This connection is explained in
$\S$\ref{eq:homomorphism-to-representation}. We then pass to dynamical
considerations by the relationship between Selberg zeta functions
and dynamical zeta functions explained in $\S\ref{subsec:Zeta-functions}.$
The relevant dynamical zeta functions are Fredholm determinants of
certain transfer operators on vector valued functions, twisted by
(random) unitary representations $\rho_{n}^{0}$ of $\Gamma$. These
are introduced in $\S$\ref{subsec:Functional-spaces-and}. The relevance
of these representations is that the zeros of the $\rho_{n}^{0}$-twisted
Selberg zeta function of $X$ correspond to new resonances of $X_{n}$
(see $\S$\ref{subsec:Zeta-functions}). These are precisely the objects
we wish to control.

\textbf{Theorems \ref{thm:L2-gap} and \ref{thm:rectangle-gap}. }Since
Theorem \ref{thm:rectangle-gap} implies Theorem \ref{thm:L2-gap}
it suffices to discuss the former. 

So far we have not been precise about the transfer operators we use.
To prove Theorem \ref{thm:rectangle-gap} we do not use the `standard'
twisted transfer operators used for example in \cite{BGS2,OW,MOW},
but rather, we base our twisted operators on the \emph{refined transfer
operators} introduced by Bourgain and Dyatlov in \cite{BDFourier}.
The operators are denoted by $\L_{\tau,s,\rho_{n}^{0}}$ and defined
precisely in $\S$\ref{subsec:Functional-spaces-and}. The parameter
$s$ is a frequency parameter, and the parameter $\tau$ is a `discretization
parameter' that is taken to be $n^{-\frac{2}{\delta}}$. If we do
not use this operator in the definition of the dynamical zeta function,
but rather, an iterate of the standard one, without the built in parameter
$\tau$, then one can still follow the strategy of this paper to obtain
resonance-free regions. \uline{However, these will depend on subtle
features of the graph of the pressure functional \mbox{$P(\sigma)$}
defined in \mbox{$\S$}\mbox{\ref{subsec:Words,-encodings-of}}.}
It is the use of refined transfer operators that allows us to improve
on this, and is a key idea in the paper. The functional spaces we
use are Bergman spaces, and this gives us crucial access to trace
techniques.

To control zeros of the dynamical zeta function in a rectangle, we
use Jensen's formula with a circle enclosing the rectangle (cf. Figure
\ref{fig:jensen}). The strategy is to prove that the expected number
of zeros in the region decays as a polynomial in $n$, so by Markov's
inequality, a.a.s. there are none. There are two terms in Jensen's
formula we need to control. The first is $\log|\det(1-\L_{\tau,s,\rho_{n}^{0}})|$
when $s$ is the center of the circle. As shown in Proposition \ref{prop:pointiwse-bound},
this term decays provided the center of the circle is a sufficiently
large real number, which can be arranged. The second term in Jensen's
formula is the integral over $s$ in the circle of $\log|\det(1-\L_{\tau,s,\rho_{n}^{0}}^{2})$\textbar .
A convenient property of Jensen's formula is that it is an integral
formula, and we can take expectations inside the integral. Using Weyl's
inequality, and taking expectations, we reduce to bounding the expectation
$\E_{n}\|\L_{\tau,s,\rho_{n}^{0}}\|_{\HS}^{2}$ of the squared Hilbert-Schmidt
norm of $\L_{\tau,s,\rho_{n}^{0}}$ for $s$ on the circle. We need
to prove these all decay uniformly and polynomially in $n$. This
estimate is at the core of the proof, is stated precisely in Proposition
\ref{prop:main-prob-estimate}, and its proof takes up $\S\ref{sec:The-expectation-of}$.

We now discuss the proof of Proposition \ref{prop:main-prob-estimate}.
The first step is a formula for $\E_{n}\|\L_{\tau,s,\rho_{n}^{0}}\|_{\HS}^{2}$.
This uses a deterministic expression for $\|\L_{\tau,s,\rho_{n}^{0}}\|_{\HS}^{2}$
involving a Bergman kernel and given in Lemma \ref{lem:expression-for-HS-norm}.
The formula for $\|\L_{\tau,s,\rho_{n}^{0}}\|_{\HS}^{2}$ is a complex
weighted sum of random variables $\Tr[\rho_{n}^{0}(\gamma_{\a'}\gamma_{\b'}^{-1})]$,
where $\gamma_{\a'}$ and $\gamma_{\b'}$ are elements of $\Gamma$.
By linearity of expectations we obtain an expression for $\E_{n}\|\L_{\tau,s,\rho_{n}^{0}}\|_{\HS}^{2}$
as a weighted sum of expectations
\begin{equation}
\E_{n}[\Tr(\rho_{n}^{0}(\gamma_{\a'}\gamma_{\b'}^{-1}))]\label{eq:expected-trace=000023}
\end{equation}
By passing to a majorant, in Lemma \ref{lem:maj-exp-hs-norm} we reduce
our task to estimating a sum of the form
\begin{equation}
\sum_{\a,\b\in\ZZ(\tau)}|\E_{n}[\Tr(\rho_{n}^{0}(\gamma_{\a'}\gamma_{\b'}^{-1}))]|\label{eq:exponential}
\end{equation}
where $\ZZ(\tau)$ is a set of words in the generators of $\Gamma$,
and $\a'$ is $\a$ with the last letter removed.

The strategy is to insert good bounds for (\ref{eq:expected-trace=000023})
into (\ref{eq:exponential}) to obtain the decay we want. This is
analogous to the \emph{trace method }used to bound the spectral gap
of a random graph, where $\|\L_{\tau,s,\rho_{n}^{0}}\|_{\HS}^{2}$
would be replaced by the trace of a power of the adjacency matrix.
Indeed, the bounds we use for (\ref{eq:expected-trace=000023}) go
back to the paper of Broder and Shamir \cite{BroderShamir} who used
the trace method to show that the second largest eigenvalue of a $2k$-regular
random graph in the permutation model is a.a.s. $\leq3k^{\frac{3}{4}}$.
So the appearance of $\frac{3}{4}$ in Theorems \ref{thm:L2-gap}
and \ref{thm:rectangle-gap} is similar to \emph{(ibid.).}

In \emph{(ibid.) }Broder and Shamir proved, roughly speaking, that
$\E_{n}[\Tr(\rho_{n}^{0}(\gamma))]$ has a trivial bound if $\gamma$
is the identity, a better bound if $\gamma$ is a proper power of
another element in $\Gamma$, and an even better bound if $\gamma$
does not fall in one of the previous two cases. We need a two sided
estimate for (\ref{eq:expected-trace=000023}) that can be deduced
from more recent work of Puder \cite{PUDER} and is stated in Theorem
\ref{thm:broder-shamir-puder}. According to the three cases above,
we partition the range of summation in (\ref{eq:exponential}) into
three different sets.

The hardest of these to deal with in (\ref{eq:exponential}) is the
set $\Pairs(\tau)$ that consists of $\a,\b\in\ZZ(\tau)$ such that
$\gamma_{\a'}\gamma_{\b'}^{-1}$ is a proper power in $\Gamma$. We
need to show that the contribution of this set to (\ref{eq:exponential})
has polynomial decay. We give a precise bound on $|\Pairs(\tau)|$
in Proposition \ref{prop:key-prop}; this proposition is at the core
of the paper so we now explain the ideas of its proof.

Throughout the paper we work with real quantities $\Upsilon_{\a}$,
where $\a$ is a word in the generators of $\Gamma$. These are defined
in $\S$\ref{sec:Estimates-for-Derivatives}. Roughly speaking, $\Upsilon_{\a}$
measures the size of the derivative of the associated group element
$\gamma_{\a'}$, and the set $\bar{Z}(\tau)$ is the set of words
$\a$ such that $\Upsilon_{\a}\approx\tau$. This means that estimating
$|\Pairs(\tau)|$ is roughly the same as estimating the sum
\begin{equation}
\sum_{(\a,\b)\in\Pairs(\tau)}\Upsilon_{\a}^{\frac{\delta}{2}}\Upsilon_{\b}^{\frac{\delta}{2}};\label{eq:upsilon-sum}
\end{equation}
the choice of the exponent $\frac{\delta}{2}$ optimizes the result
we can get from this method. The key combinatorial observation we
use to estimate (\ref{eq:upsilon-sum}) is that if $\gamma_{\a'}\gamma_{\b'}^{-1}$
is a proper power, after performing an absolutely bounded finite number
of the following operations
\begin{itemize}
\item cutting the sequences $\a'$ and $\b'$,
\item possibly replacing some cut sequence with its `mirror',
\item and regluing
\end{itemize}
one can form a long identical pair of sequences. This idea is performed
rigorously in $\S$\ref{subsec:Estimating-sigma2}. The result of
these operations on the $\Upsilon$ is to introduce a bounded multiplicative
constant, since $\Upsilon$ is roughly multiplicative (Lemma \ref{lem:coarse-homomorphism})
and behaves well with respect to mirrors (Lemma \ref{lem:mirror-1}).
The result of obtaining the long identical pair of sequences is that
we get bounds on (\ref{eq:upsilon-sum}) from the relationship between
sums of $\Upsilon_{\a}$ and the pressure functional (Lemma \ref{lem:Pressure-estimate}).

\textbf{Theorem \ref{thm:high-frequency}. }The proof of Theorem \ref{thm:high-frequency}
is given in $\S$\ref{sec:Unitary-representations-and}. It is based
on uniform Dolgopyat estimates for arbitrary unitary representations
of $\Gamma$. We use the main result of Bourgain and Dyatlov \cite{BDFourier}
on Patterson-Sullivan measures and Fourier decay to provide a short
and completely general proof of the uniform Dolgopyat estimates without
having to rely on the more difficult technique from \cite{NaudExpanding},
which was also used in \cite{OW,MOW}.

\subsection{Notation}

If $U\subset\C$ we write $\overline{U}$ for the closure of $U$.
We write $\N$ for the natural numbers and $\N_{0}=\N\cup\{0\}$.

\subsection{Acknowledgments}

We thank Benoît Collins and Doron Puder for helpful conversations
related to this project. Both authors thank Semyon Dyatlov for discussions
around this subject and the hospitality of IAS while attending the
conference ``Emerging Topics: Quantum Chaos and Fractal Uncertainty
Principle'' in Fall 2017. FN is supported by Institut Universitaire
de France. We thank the anonymous referee for several comments that
have improved the paper.

\section{Preliminaries}

In this paper we use the notational system for Schottky groups that
is used in the papers of Dyatlov and Bourgain \cite{BDFourier} and
Dyatlov and Zworski \cite{DZFractal} since it is very convenient
for the analysis in the sequel. We follow these papers closely in
our development.

\subsection{Words, encodings of Schottky groups, and pressure\label{subsec:Words,-encodings-of}}

Let $r\geq2$ and $\A=\{1,\ldots,2r\}$. If $a\in\A$, then we write
$\bar{a}=a+r\bmod2r$. The setup of our paper is that we are given
for each $a\in\A$ an \emph{open}\footnote{This is a difference from the notation of \cite{BDFourier} that we
make the reader aware of.} disc $D_{a}$ in $\C$ with center in $\R$. The closures of the
discs $D_{a}$ for $a\in\A$ are assumed to be disjoint from one another.
We let $I_{a}=D_{a}\cap\R$, an open interval. We write $\D=\cup_{a\in\A}D_{a}$
for the union of the discs.

We consider the usual action of $\SL_{2}(\R)$ by Möbius transformations
on the extended complex plane $\hat{\C}=\C\cup\{\infty\}$. We are
given for each $a\in\A$ a matrix $\gamma_{a}\in\SL_{2}(\R)$ with
the properties 
\[
\gamma_{a}\left(\hat{\C}-D_{\bar{a}}\right)=\overline{D_{a}},\quad\gamma_{\bar{a}}=\gamma_{a}^{-1}.
\]

\begin{figure}
\includegraphics[scale=0.75]{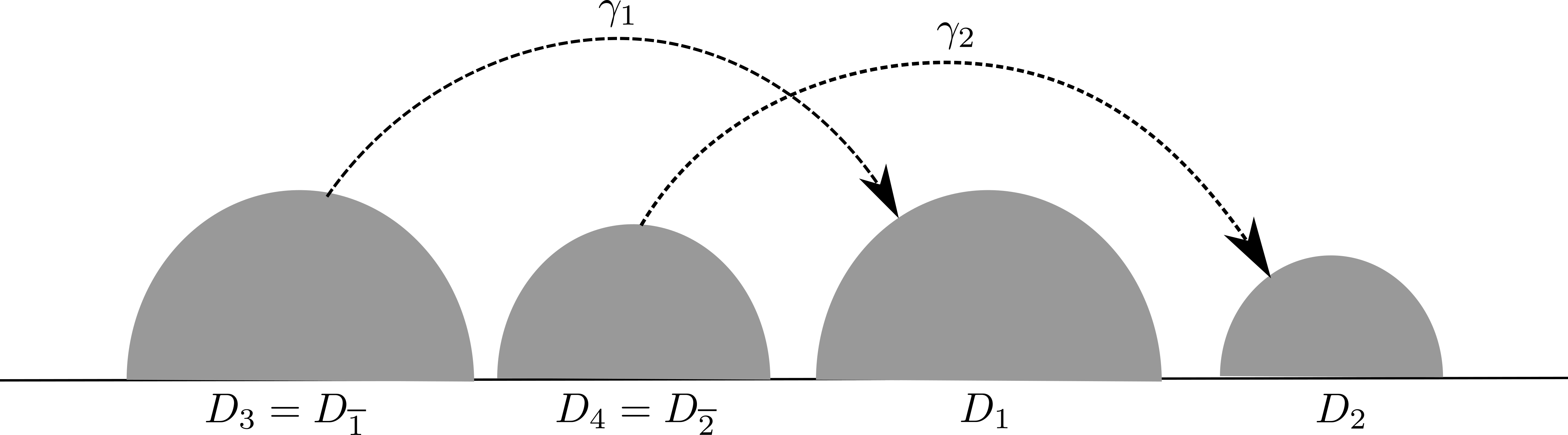}

\caption{An example of Schottky pairing with $r=2$}

\end{figure}
We write $\Gamma=\langle\gamma_{a}\,:\,a\in\A\rangle$ for the group
generated by the $\gamma_{a}$. Since the $D_{a}$ are disjoint, the
Ping-Pong Lemma shows that $\Gamma$ is a free subgroup of $\SL_{2}(\R)$.
Any group obtained by this construction is called a \emph{Schottky
group. }It is a result of Button \cite{Button} that if $X=\Gamma\backslash\mathbb{H}$
is a connected convex co-compact Riemann surface as in our main theorems,
then $\Gamma$ is a Schottky group; we now fix $\Gamma$ and assume
it arises from the above construction.

The elements of $\Gamma$ can be encoded by words in the alphabet
$\A$ as follows. A \emph{word }is a finite sequence
\[
\a=(a_{1},\ldots,a_{n}),\quad n\in\N\cup\{0\}
\]
such that $a_{i}\neq\overline{a_{i+1}}$ for $i=1,\ldots,n-1$. We
say that $n$ is the\emph{ length} of $\a$ and denote this by $|\a|=n$.
We write $\W$ for the collection of all words, $\W_{N}$ for the
words of length $N$, and $\W_{\geq N}$ for the words of length $\geq N$.
We write $\emptyset$ for the empty word and write $\W^{\circ}=\W-\{\emptyset\}$.
For $\a=(a_{1},\ldots,a_{n}),\b=(b_{1},\ldots b_{m})\in\W$ we write
\begin{itemize}
\item ${\bf a'}=(a_{1},\ldots,a_{n-1})$ if $\a=(a_{1},\ldots,a_{n})$ and
$n\geq1$.
\item $\a\to\b$ if either of $\a$ or $\b$ is empty, or else $a_{n}\neq\overline{b_{1}}$,
in which case $(a_{1},\ldots,a_{n},b_{1},\ldots,b_{m})$ is in $\W^{\circ}$
and we write $\a\b$ for this concatenation.
\item $\a\rsa\b$ if $\a,\b\in\W^{\circ}$ and $a_{n}=b_{1}$, which case
$\a'\b$ is in $\W^{\circ}$. 
\end{itemize}
If $\a=(a_{1},\ldots,a_{n})\in\W$ then we associate to $\a$ the
group element $\gamma_{\a}\eqdf\gamma_{a_{1}}\ldots\gamma_{a_{n}}$;
here $\gamma_{\emptyset}=\mathrm{id}.$ The map $\a\in\W\mapsto\gamma_{\a}\in\Gamma$
is a one-to-one encoding of $\Gamma$. We write $\overline{\a}\eqdf(\overline{a_{n}},\ldots,\overline{a_{1}})$
and call this the \emph{mirror} of $\a$. Note that $\gamma_{\overline{\a}}=\gamma_{\a}^{-1}$.
If $\a=(a_{1},\ldots,a_{n})\in\W^{\circ}$ we let 
\[
D_{\a}=\gamma_{\a'}(D_{a_{n}}),\quad I_{\a}=\gamma_{\a'}(I_{a_{n}})
\]
and write $|I_{\a}|$ for the length of the open interval $I_{\a}$.

The \emph{Bowen-Series map} $T:\D\to\hat{\C}$ is given by 
\[
T\lvert_{D_{a}}=\gamma_{a}^{-1}=\gamma_{\bar{a}}.
\]
The Bowen-Series map is eventually expanding \cite[Prop. 15.5]{Borthwick};
this will be made explicit below so we do not give the general definition
now. The limit set $\Lambda=\Lambda(\Gamma)$ of $\Gamma$, defined
in the Introduction, coincides with the non-wandering set of $T$:
\[
\Lambda(\Gamma)=\bigcap_{n=1}^{\infty}T^{-n}(\D).
\]
The limit set $\Lambda$ is a compact $T$-invariant subset of $\R$.
Given a Hölder continuous map $\varphi:\Lambda\rightarrow\R$, the
\emph{topological pressure} $P(\varphi)$ can be defined through the
variational formula: 
\[
P(\varphi)=\sup_{\mu}\left(h_{\mu}(T)+\int_{\Lambda}\varphi d\mu\right),
\]
where the supremum is taken over all $T$-invariant probability measures
on $\Lambda$, and $h_{\mu}(T)$ stands for the measure-theoretic
entropy. A celebrated result of Bowen \cite{Bowen} says that the
map 
\[
\sigma\mapsto P(-\sigma\log\vert T'\vert)
\]
is convex\footnote{Convexity follows obviously from the variational formula above. },
strictly decreasing and vanishes exactly at $\sigma=\delta(\Gamma)$,
the Hausdorff dimension of the limit set $\Lambda$. In addition,
it is not difficult to see from the variational formula that $P(-\sigma\log\vert T'\vert)$
tends to $-\infty$ as $\sigma\rightarrow+\infty$. For simplicity,
we will use the notation $P(\sigma)$ in place of $P(-\sigma\log\vert T'\vert)$.
The pressure will play a role in some of the estimates in the sequel.

\subsection{Functional spaces and transfer operators\label{subsec:Functional-spaces-and}}

Let $V$ be any Hilbert space. If $\Omega$ is any open subset of
the complex numbers $\C$, we consider the Bergman space $\H(\Omega;V)$
that is the space of $V$-valued holomorphic functions on $\Omega$
with finite norm with respect to the given inner product
\[
\langle f,g\rangle\stackrel{\df}{=}\int_{\Omega}\langle f(x),g(x)\rangle_{V}dm(x).
\]
Here $dm$ is Lebesgue measure on $\Omega$. If $V$ is separable,
then $\H(\Omega;V)$ is a separable Hilbert space; in this paper $V$
will always be finite dimensional.

Of particular interest is $\H(\D;V)$. This splits as an orthogonal
direct sum
\[
\H(\D;V)=\bigoplus_{a\in\A}\H(D_{a};V).
\]
If $\{\e_{k}\}_{k=1}^{\infty}$ is any orthonormal basis of $\H(D_{a};\C)$,
and $x_{1},x_{2}\in D_{a}$, then the sum
\[
\sum_{k=1}^{\infty}\e_{k}(x_{1})\overline{\e_{k}(x_{2})}\eqdf B_{D_{a}}(x_{1},x_{2})
\]
converges and the resulting kernel is called the \emph{Bergman kernel
of $D_{a}$. }It is given by the explicit formula (cf. \cite[pg. 378]{Borthwick})
\begin{equation}
B_{D_{a}}(x_{1},x_{2})=\frac{r_{a}^{2}}{\pi\left[r_{a}^{2}-(\overline{x}_{2}-c_{a})(x_{1}-c_{a})\right]^{2}}\label{eq:explicit-bergman}
\end{equation}
where $r_{a},c_{a}$ are the radius and center of $D_{a}$.

Throughout the sequel, $\rho:\Gamma\to\U(V)$ will be a unitary representation
of the Schottky group $\Gamma$. If $Z\subset\W^{\circ}$ is any finite
subset of words, then we define
\begin{align*}
\L_{Z,s,\rho}[f](x) & =\sum_{\substack{\a\in Z\\
\a\rsa b
}
}\gamma'_{\a'}(x)^{s}\rho(\gamma_{\a'}^{-1})f(\gamma_{\a'}(x))\quad x\in D_{b},b\in\A.
\end{align*}
The complex power $\gamma'_{\a'}(x)^{s}$ is defined by analytic continuation
using that $\gamma'_{\a'}(x)$ is positive on $I_{b}$ and never a
negative real on $D_{b}$. One has $\L_{Z,s,\rho}:\H(\D;V)\to\H(\D;V)$.
Certain particular choices of $Z$ are made throughout the paper.
The basic type of transfer operator that is considered corresponds
to the choice $Z=\W_{2}$. We write $\L_{s,\rho}\eqdf\L_{\W_{2},s,\rho}$.
This operator can be written as
\[
\L_{s,\rho}[f](x)=\sum_{\substack{a\in\A\\
a\to b
}
}\gamma'_{a}(x)^{s}\rho(\gamma_{a}^{-1})f(\gamma_{a}(x))\quad x\in D_{b},b\in\A.
\]
In the following we follow Dyatlov and Zworski \cite[\S 2.4]{DZFractal}.
\begin{defn}
A subset $Z\subset\W^{\circ}$ is a \emph{partition} if there is $N\geq0$
such that for all $\a\in\W$ with $|\a|\geq N$, there is a unique
$\b\in Z$ that is a prefix of $\a$.
\end{defn}

One particular family of partitions, introduced by Bourgain and Dyatlov
\cite{BDFourier}, plays an important role in this paper. For any
$\tau>0$ we define
\[
Z(\tau)\stackrel{\df}{=}\{\,\a\in\W^{\circ}:\,|I_{\a}|\leq\tau<|I_{\a'}|\,\}.
\]
It is shown by Dyatlov and Zworski \cite[eqs (2.7), (2.15)]{DZFractal}
that this is indeed a partition. Not only is the partition $Z(\tau)$
important to us, but so too is its mirror set
\[
\overline{Z}(\tau)\stackrel{\df}{=}\{\,\a\in\W^{\circ}:\,\overline{\a}\in Z(\tau)\,\}.
\]
The reason for introducing this mirror set is to make Lemma \ref{lem:fixed-points-of-transfer-operator-passage}
below work. Note that $\ZZ(\tau)$ may not be a partition, although
this will not matter. We write $\L_{\tau,s,\rho}\eqdf\L_{\ZZ(\tau),s,\rho}$.

\subsection{The representations appearing in this paper\label{subsec:The-representations-appearing}}

In this paper we consider particular types of representations $\rho:\Gamma\to\U(V)$
as follows. We consider $n\in\N$ and the family of symmetric groups
$S_{n}$ on $n$ letters. Let $V_{n}\eqdf\ell^{2}(\{1,\ldots,n\})$.
The group $S_{n}$ has a standard representation $\std_{n}:S_{n}\to\U(V_{n})$
where $S_{n}$ acts by precomposition on $\ell^{2}$ functions $f:\{1,\ldots,n\}\to\C$.
This representation is not irreducible, but splits as an orthogonal
direct sum $\mathbf{1}\oplus V_{n}^{0}$ where $V_{n}^{0}$ is an
irreducible representation of dimension $n-1$. We write $\std_{n}^{0}:S_{n}\to\U(V_{n}^{0})$
for the corresponding homomorphism of the symmetric group.

We now build a representation from a homomorphism $\phi_{n}:\Gamma\to S_{n}$.
Since $\Gamma$ is free, $\phi_{n}$ is described simply by choosing
the images of a generating set of $\Gamma$, which may be taken to
be the $\gamma_{a}$ with $1\leq a\leq r$. We consider 
\begin{equation}
\rho_{n}\eqdf\std_{n}\circ\phi_{n},\quad\rho_{n}^{0}\eqdf\std_{n}^{0}\circ\phi_{n}.\label{eq:homomorphism-to-representation}
\end{equation}
These depend on the choice of $\phi_{n}$. Later in the paper we will
view $\phi_{n}:\Gamma\to S_{n}$ as a random homomorphism; its law
is described by choosing the $\phi_{n}(\gamma_{a})$ with $1\leq a\leq r$
independently and uniformly at random with respect to the uniform
measure on $S_{n}$. This gives random representations $\rho_{n}$
and $\rho_{n}^{0}$. We write $\E_{n}$ to refer to expectations of
random variables with repect to the random representation $\rho_{n}^{0}$.
For example, if $\gamma\in\Gamma$, then $\Tr[\rho_{n}^{0}(\gamma)]$
is a real random variable and we write $\E_{n}(\Tr[\rho_{n}^{0}(\gamma)])$
for its expectation. At other times we view $\phi_{n}$, $\rho_{n}$,
$\rho_{n}^{0}$ as fixed and coupled to one another; it will be clear
from the context whether we make probabilistic or deterministic statements.

\subsection{Selberg zeta functions\label{subsec:Selberg-zeta-functions}}

If $X$ is any convex co-compact hyperbolic surface (not necessarily
connected), then the \emph{Selberg zeta function} of $X$ is defined
for $\mathrm{Re}(s)>\delta$ by
\[
Z_{X}(s)\eqdf\prod_{\gamma\in\P(X)}\prod_{k=0}^{\infty}\left(1-e^{-(s+k)l(\gamma)}\right)
\]
where $\P(X)$ is the collection of primitive\footnote{Primitive here means it is not an iterate of a shorter closed geodesic.}
closed geodesics on $X$, and $l(\gamma)$ is the length of such a
geodesic. The function $Z_{X}(s)$ analytically continues to an entire
function \cite{Guillope,GLZ}. One has the following theorem due to
Patterson and Perry \cite[Theorem 1.5]{PattersonPerry} relating resonances
of the Laplacian to the Selberg zeta function.
\begin{thm}[Patterson-Perry]
\label{thm:Patterson-Perry}If $X$ is any non-elementary convex
co-compact hyperbolic surface, then any resonance of $X$ is a zero
of $Z_{X}$. Conversely, if $s$ is a zero of $Z_{X}$ with $\mathrm{Re}(s)>0$
then $s$ is a resonance of $X$. In all cases, the order of the zero
of $Z_{X}$ is equal to the multiplicity of the corresponding resonance.
\end{thm}

We will also have a use for twisted Selberg zeta functions. If $\rho:\Gamma\to\U(V)$
is any finite dimensional unitary representation of $\Gamma$ then
we let 
\[
Z_{X,\rho}(s)\eqdf\prod_{\gamma\in\P(X)}\prod_{k=0}^{\infty}\det\left(1-\rho(\gamma)e^{-(s+k)l(\gamma)}\right).
\]
This converges to a holomorphic function in $\mathrm{Re}(s)>\delta$
and extends to an entire function by results in \cite{FP}.

\section{Estimates for derivatives\label{sec:Estimates-for-Derivatives}}

The following section contains certain technical but either easy or
well-known estimates for derivatives of $\Gamma$ that will be used
in the sequel. The fundamental estimates for derivatives of elements
of $\Gamma$ are the following:
\begin{lem}
\label{lem:standard-estimates}~
\begin{description}
\item [{Uniform~contraction}] There are $C=C(\Gamma)>0$ and $0<\bar{\theta}<\theta<1$
such that for all $\a\in\W$, $b\in\A$ with $\a\to b$, and $x\in D_{b}$,
\begin{equation}
C^{-1}\bar{\theta}^{|\a|}\leq|\gamma_{\a}'(x)|\leq C\theta^{|\a|}.\label{eq:contraction}
\end{equation}
\item [{Bounded~distortion~I}] There is $K=K(\Gamma)>0$ such that for
all $b\in\A$, $\a\in\W$ such that $\a\to b$ and all $x_{1},x_{2}\in D_{b}$,
\begin{equation}
e^{-|x_{1}-x_{2}|K}\leq\frac{|\gamma_{\a}'(x_{1})|}{|\gamma_{\a}'(x_{2})|}\leq e^{|x_{1}-x_{2}|K}.\label{eq:fine-distortion}
\end{equation}
\item [{Bounded~distortion~II}] There is a constant $c=c(\Gamma)>0$
such that for $\a\in\W$, $b_{1},b_{2}\in\A$ with $\a\to b_{1},b_{2}$
and $x_{1}\in D_{b_{1}}$, $x_{2}\in D_{b_{2}}$,
\begin{equation}
\frac{|\gamma_{\a}'(x_{1})|}{|\gamma_{\a}'(x_{2})|}\leq c.\label{eq:uniform-distortion}
\end{equation}
\end{description}
\end{lem}

\begin{proof}
The first two properties can be found in \cite[\S 2]{Naud}. The last
part is trivial if $\a=\emptyset$. Otherwise, if $|\a|\geq1$ we
can write $\a=\a'a$ with $\a'\in\W$ and $\a'\to a\to b_{1},b_{2}$.
Then for $x_{i}\in D_{b_{i}}$ we have
\[
|\gamma_{\a}'(x_{i})|=\gamma_{\a'}'(\gamma_{a}(x_{i}))\gamma_{a}'(x_{i})\quad i=1,2.
\]
We have $\frac{|\gamma_{a}'(x_{1})|}{|\gamma_{a}'(x_{2})|}\leq C$
by (\ref{eq:contraction}) and since now $\gamma_{a}(x_{1})$ and
$\gamma_{a}(x_{2})$ are in $D_{a}$, (\ref{eq:fine-distortion})
gives 
\[
\frac{|\gamma_{\a'}'(\gamma_{a}(x_{1}))|}{|\gamma_{\a'}'(\gamma_{a}(x_{1}))|}\leq\exp(K\sup_{b\in\A}\mathrm{diameter}(D_{b})).
\]
The equation (\ref{eq:uniform-distortion}) now follows.
\end{proof}
In the rest of the paper, for any $\a\in\W^{\circ}$, we define
\[
\Upsilon_{\a}\eqdf|I_{\a}|.
\]
We set $\Upsilon_{\emptyset}\eqdf1$. For $\a\in\W^{\circ}$, we have
\begin{equation}
\Upsilon_{\a}\leq\Upsilon_{\a'}\label{eq:trivial-child-ineq}
\end{equation}
 since $I_{\a}\subset I_{\a'}$. Therefore there is $c=c(\Gamma)>0$
such that for any $\a\in\W$
\begin{equation}
0<\Upsilon_{\a}\leq c.\label{eq:upsilon-bounded}
\end{equation}
We next recall some useful results of Bourgain-Dyatlov from \cite[\S 2]{BDFourier}. 
\begin{lem}
\label{lem:deriv-to-upsilon}There is a constant $K_{0}=K_{0}(\Gamma)>1$
such that for any $\a=(a_{1},\ldots,a_{n})\in\W^{\circ}$ and $x\in D_{a_{n}}$
\[
K_{0}^{-1}\Upsilon_{\a}\leq|\gamma'_{\a'}(x)|\leq K_{0}\Upsilon_{\a}.
\]
\end{lem}

\begin{proof}
For $x\in I_{a_{n}}$ this is \cite[Lemma 2.5, (20)]{BDFourier}.
The more general result here follows by combining \cite[Lemma 2.5]{BDFourier}
with the bounded distortion estimate (\ref{eq:uniform-distortion}).
\end{proof}
The following lemma is \cite[Lemma 2.10, (30)]{BDFourier}.
\begin{lem}
\label{lem:upsilon-to-tau}There is a constant $K_{1}=K_{1}(\Gamma)>1$
such that for $\tau\in(0,1)$, for any $\a\in\ZZ(\tau)$ we have
\[
K_{1}^{-1}\tau\leq\Upsilon_{\a}\leq K_{1}\tau.
\]
\end{lem}

The next lemma says that $\Upsilon$ is coarsely multiplicative.
\begin{lem}
\label{lem:coarse-homomorphism}There is a constant $K_{2}=K_{2}(\Gamma)>1$
such that for all $\a,\b\in\W^{\circ}$ with $\a\rsa\b$
\[
K_{2}^{-1}\Upsilon_{\a}\Upsilon_{\b}\leq\Upsilon_{\a'\b}\leq K_{2}\Upsilon_{\a}\Upsilon_{\b},
\]
and for $\a,\b\in\W$ with $\a\to\b$
\begin{equation}
K_{2}^{-1}\Upsilon_{\a}\Upsilon_{\b}\leq\Upsilon_{\a\b}\leq K_{2}\Upsilon_{\a}\Upsilon_{\b}.\label{eq:mult-pure}
\end{equation}
\end{lem}

\begin{proof}
The first set of inequalities is \cite[Lemma 2.7]{BDFourier}. If
either $\a$ or $\b$ is $\emptyset$, then (\ref{eq:mult-pure})
is trivially true with $K_{2}=1$. So assume $\a,\b\in\W^{\circ}$.
Then (\ref{eq:mult-pure}) follows by combining \cite[Lemmas 2.6 and 2.7]{BDFourier}. 
\end{proof}
We also have the following `mirror' estimate for $\Upsilon$.

\begin{lem}[{Mirror estimate, \cite[Lemma 2.8]{BDFourier}}]
\label{lem:mirror-1}There is a constant $K_{3}=K_{3}(\Gamma)>1$
such that for any $\a\in\W$ 
\[
K_{3}^{-1}\Upsilon_{\overline{\a}}\leq\Upsilon_{\a}\leq K_{3}\Upsilon_{\overline{\a}}.
\]
\end{lem}

We now state some lemmas about the set $\ZZ(\tau)$.
\begin{lem}
\label{lem:deriv-to-tau}There is a constant $C_{1}=C_{1}(\Gamma)>1$
such that for $\a=(a_{1},\ldots,a_{n})\in\ZZ(\tau)$, for any $x\in D_{a_{n}}$
we have 
\[
C_{1}^{-1}\tau\leq|\gamma'_{\a'}(x)|\leq C_{1}\tau.
\]
\end{lem}

\begin{proof}
This follows by combining Lemmas \ref{lem:deriv-to-upsilon} and \ref{lem:upsilon-to-tau}.
\end{proof}
Given Lemma \ref{lem:deriv-to-tau}, we can make the following estimate
on the word lengths of elements $\a\in\ZZ(\tau).$
\begin{lem}
\label{lem:tau-to-word-length}There are constants $D=D(\Gamma)>1$
and $\kappa=\kappa(\Gamma)>0$ such that if $\a\in\ZZ(\tau)$, then
\[
D^{-1}\log\tau^{-1}-\kappa\leq|\a|\leq D\log\tau^{-1}+\kappa.
\]
\end{lem}

\begin{proof}
Write $\a=(a_{1},\ldots,a_{n})$. Pick $x\in D_{a_{n}}$. By Lemma
\ref{lem:deriv-to-tau} we have
\[
C_{1}^{-1}\tau\leq|\gamma'_{\a'}(x)|\leq C_{1}\tau,
\]
and combining this with (\ref{eq:contraction}) gives
\[
C_{1}^{-1}C^{-1}\bar{\theta}^{|\a'|}\leq\tau\leq CC_{1}\theta^{|\a'|}.
\]
Since $|\a|=|\a'|+1$, this gives the result after taking logarithms
and rearranging.
\end{proof}
We now note
\begin{lem}
\label{lem:Zbar-is-good}There is $0<\tau_{0}<1$ such that for $\tau<\tau_{0}$,
$\ZZ(\tau)\subset\W_{\geq2}$.
\end{lem}

\begin{proof}
This is a direct consequence of Lemma \ref{lem:tau-to-word-length}.
\end{proof}
\emph{Throughout the sequel, $\tau_{0}$ will always be the parameter
given by Lemma \ref{lem:Zbar-is-good}}. It will also be useful to
know roughly how many elements there are in $\ZZ(\tau)$. This is
given by \cite[Lemma 2.13]{BDFourier} (noting that $|\ZZ(\tau)|=|Z(\tau)|$).
\begin{lem}
\label{lem:size-of-partition}There is $C_{2}=C_{2}(\Gamma)>1$ such
that for $\tau\in(0,1]$ 
\[
C_{2}^{-1}\tau^{-\delta}\leq|\ZZ(\tau)|\leq C_{2}\tau^{-\delta}.
\]
\end{lem}

To conclude this section, we record that certain sums of derivatives
are related to the pressure functional. 
\begin{lem}
\label{lem:Pressure-estimate}For all $\sigma_{1},Q\in\R$ such that
$0\leq\sigma_{1}<Q$ there is a constant $C=C(\sigma_{1},Q)>0$ such
that for all $N\in\N_{0}$ and $\sigma\in[\sigma_{1},Q]$ we have
\begin{equation}
\sum_{a\in\A}\sum_{\substack{\a\in\W_{N}\\
\a\rsa a
}
}\sup_{I_{a}}|\gamma_{\a'}|^{\sigma}\leq C\exp(NP(\sigma_{1})),\label{eq:pressure-est-1}
\end{equation}
and
\begin{equation}
\sum_{\a\in\W_{N}}\Upsilon_{\a}^{\sigma}\leq C\exp(NP(\sigma_{1})).\label{eq:pressure-est-2}
\end{equation}
\end{lem}

\begin{proof}
The estimate (\ref{eq:pressure-est-1}) is a standard estimate that
appears in \cite[Lemma 3.1]{Naud}. The estimate (\ref{eq:pressure-est-2})
follows by combining (\ref{eq:pressure-est-1}) with Lemma \ref{lem:deriv-to-upsilon}
and increasing $C$.
\end{proof}

\section{Transfer operators and zeta functions}

\subsection{Zeta functions\label{subsec:Zeta-functions}}
\begin{lem}
For any $Z\subset\W_{\geq2}$, and any finite dimensional unitary
representation $\rho$ of $\Gamma$, the operator $\L_{Z,s,\rho}$
is trace class on $\H(\D;V)$.
\end{lem}

\begin{proof}
The proof is an easy adaptation of \cite[Lemma 15.7]{Borthwick}.
The condition $Z\subset\W_{\geq2}$ rules out $\L_{Z,s,\rho}$ having
any summand that acts as the identity on some $D_{a}$.
\end{proof}
\begin{cor}
\label{cor:trace-class-operators}Let $(\rho,V)$ be any finite dimensional
unitary representation of $\Gamma$.
\begin{enumerate}
\item The operator $\L_{s,\rho}$ is trace class on $\H(\D;V)$.
\item For $\tau<\tau_{0}$, the operator $\L_{\tau,s,\rho}$ is trace class
on $\H(\D;V)$.
\end{enumerate}
\end{cor}

Given Corollary \ref{cor:trace-class-operators} we can define \emph{zeta
functions}
\begin{align*}
\zeta_{\rho}(s) & \eqdf\det(1-\L_{s,\rho}),\\
\zeta_{\tau,\rho}(s) & \eqdf\det(1-\L_{\tau,s,\rho}^{2}).
\end{align*}
The determinants that appear here are Fredholm determinants. The reason
that we have used $\L_{\tau,s,\rho}^{2}$ in the definition of $\zeta_{\tau,\rho}(s)$
is that it will later allow us to estimate $\log|\zeta_{\tau,\rho}(s)|$
in terms of the Hilbert-Schmidt norm of $\L_{\tau,s,\rho}$ rather
than the trace norm (cf. (\ref{eq:first-est})). On the other hand,
we do not square $\L_{s,\rho}$ in the definition of $\zeta_{\rho}(s)$
so that we can access known results about $\zeta_{\rho}(s)$. 

By the general theory of Fredholm determinants we have 
\begin{lem}
\label{lem:zeros-and-fixed-points}Let $(\rho,V)$ be any finite dimensional
unitary representation of $\Gamma$.
\begin{enumerate}
\item The function $\zeta_{\rho}(s)$ is an entire function of $s\in\C$
and 
\[
\zeta_{\rho}(s)=0\iff\exists\,u\in\H(\D;V)\,:\,\L_{s,\rho}u=u.
\]
\item If $\tau<\tau_{0}$ then $\zeta_{\tau,\rho}(s)$ is an entire function
of $s\in\C$ and
\[
\zeta_{\tau,\rho}(s)=0\iff\exists\,u\in\H(\D;V)\,:\,\L_{\tau,s,\rho}^{2}u=u.
\]
\end{enumerate}
\end{lem}

The relevance of the zeta functions $\zeta_{\rho}(s)$ are the following:
\begin{prop}
\label{prop:factorization-induction}Let $\phi_{n}:\Gamma\to S_{n}$
be a fixed homomorphism, and $(\rho_{n},V_{n})$ the unitary representation
corresponding to $\phi_{n}$ via (\ref{eq:homomorphism-to-representation}).
Let $X_{n}$ be the $n$-cover of $X$ corresponding to $\phi_{n}$.
\begin{enumerate}
\item We have $\zeta_{\rho_{n}}(s)=Z_{X,\rho_{n}}(s)=Z_{X_{n}}(s)$.
\item We have $Z_{X_{n}}(s)=Z_{X}(s)\zeta_{\rho_{n}^{0}}(s).$
\end{enumerate}
\end{prop}

\begin{proof}
\emph{Proof of Part 1. }A special case of a result of Jakobson, Naud,
and Soares \cite[Prop. 2.2]{JNS} for arbitrary finite-dimensional
unitary representations gives 
\[
\zeta_{\rho_{n}}(s)=Z_{X,\rho_{n}}(s)
\]
where both sides are entire functions of $s$. 

If $X_{n}$ is connected, then $X_{n}=\Gamma_{n}\backslash\mathbb{H}$
for some $\Gamma_{n}\leq\Gamma$ and $\rho_{n}=\Ind_{\Gamma_{n}}^{\Gamma}1$,
the induction of the trivial representation from $\Gamma_{n}$ to
$\Gamma$. In this case the Venkov-Zograf type induction formula proved
by Fedosova and Pohl in \cite[Thm. 6.1(ii)]{FP} (cf. \cite{VZ})
gives
\[
Z_{X,\rho_{_{n}}}(s)=Z_{X_{n}}(s).
\]
If $X_{n}$ is not connected, let $X_{n}^{(1)},\ldots,X_{n}^{(m)}$
denote its connected components, and let $X_{n}^{(j)}=\Gamma_{n}^{j}\backslash\mathbb{H}$
with $\Gamma_{n}^{j}\leq\Gamma$. If we let $\rho_{n}^{j}=\Ind_{\Gamma_{n}^{j}}^{\Gamma}1$
then we have $\rho_{n}=\bigoplus_{j=1}^{m}\rho_{n}^{j}$. Then 
\[
Z_{X_{n}}(s)=\prod_{j=1}^{m}Z_{X_{n}^{(j)}}(s)=\prod_{j=1}^{m}Z_{X,\rho_{n}^{j}}(s)=Z_{X,\rho_{n}}(s)
\]
where the first equality is by definition of the Selberg zeta functions,
the second equality uses the induction formula \cite[Thm. 6.1(ii)]{FP}
and the last inequality uses the factorization formula \cite[Thm. 6.1(i)]{FP}.
Thus we have proved $\zeta_{\rho_{n}}(s)=Z_{X,\rho_{n}}(s)=Z_{X_{n}}(s)$.
\emph{This proves Part 1.}

\emph{Proof of Part 2. }Using \cite[Prop. 2.2]{JNS} again gives
\begin{equation}
\zeta_{\rho_{n}^{0}}(s)=Z_{X,\rho_{n}^{0}}(s).\label{eq:temp}
\end{equation}
Since $\rho_{n}=\mathbf{1}\oplus\rho_{n}^{0}$, we have 
\[
Z_{X_{n}}(s)=Z_{X,\rho_{n}}(s)=Z_{X}(s)Z_{X,\rho_{n}^{0}}(s)=Z_{X}(s)\zeta_{\rho_{n}^{0}}(s)
\]
where the first equality used Part 1 of the lemma, the second used
the factorization formula \cite[Thm. 6.1(i)]{FP}, and the third used
(\ref{eq:temp}). \emph{This proves Part 2.} 
\end{proof}
The following lemma adapts (a special case of) \cite[Lemma 2.4]{DZFractal}
to our vector-valued setting. The proof is essentially the same.
\begin{lem}
\label{lem:fixed-points-of-transfer-operator-passage}For all sufficiently
small $\tau>0$, if $u\in\H(\D;V)$ is such that $\L_{s,\rho}u=u$,
then\\
 $\L_{\tau,s,\rho}u=u$.
\end{lem}

\begin{cor}
\label{cor:new-zeros-zeros-of-zeta}For all sufficiently small $\tau>0$,
if $Z_{X_{n}}(s)=0$ and $Z_{X}(s)\neq0$, then $\zeta_{\tau,\rho_{n}^{0}}(s)=0$.
\end{cor}

\begin{proof}
If $Z_{X_{n}}(s)=0$, $Z_{X}(s)\neq0$, then by Proposition \ref{prop:factorization-induction},
Part 2, $\zeta_{\rho_{n}^{0}}(s)=0.$ Then by Lemma \ref{lem:zeros-and-fixed-points},
Part 1, there is $u\in\H(\D;V_{n}^{0})$ such that $\L_{s,\rho_{n}^{0}}u=u$.
By Lemma \ref{lem:fixed-points-of-transfer-operator-passage}, this
implies that $\L_{\tau,s,\rho_{n}^{0}}u=u$, and hence $\L_{\tau,s,\rho_{n}^{0}}^{2}u=u$.
Then by Lemma \ref{lem:zeros-and-fixed-points}, Part 2, $\zeta_{\tau,\rho_{n}^{0}}(s)=0$.
\end{proof}

\subsection{The Hilbert-Schmidt norm of the transfer operator}

Corollary \ref{cor:new-zeros-zeros-of-zeta} reduces controlling zeros
of the Selberg zeta function of $X_{n}$ that do not come from $X$
to controlling zeros of $\zeta_{\tau,\rho_{n}^{0}}(s)$. To do this,
we will use Jensen's formula, but before doing so, we collect some
estimates. The first will be a pointwise lower bound on $|\zeta_{\tau,\rho}(s)|$
when $s$ is a sufficiently large real number (cf. $\S$\ref{subsec:A-pointwise-estimate}).
The other will be an estimate for the expectation of the squared Hilbert-Schmidt
norm $\|\L_{\tau,s,\rho}\|_{\HS}^{2}$ for $\rho=\rho_{n}^{0}$. One
input to the latter result is a deterministic (non-random) expression
for $\|\L_{\tau,s,\rho}\|_{\HS}^{2}$ that we give now.
\begin{lem}
\label{lem:expression-for-HS-norm}Let $(\rho,V)$ be any finite dimensional
unitary representation of $\Gamma.$ We have for any $s\in\C$ and
$\tau\leq\tau_{0}$
\[
\|\L_{\tau,s,\rho}\|_{\HS}^{2}=\sum_{a,b\in\A}\sum_{\substack{\a_{1},\a_{2}\in\ZZ(\tau)\\
a\rsa\a_{1},\a_{2}\rsa b
}
}\Tr(\rho(\gamma_{\a_{1}'}\gamma_{\a'_{2}}^{-1}))\int_{D_{b}}\gamma'_{\a_{1}'}(x)^{s}\overline{\gamma'_{\a_{2}'}(x)^{s}}B_{D_{a}}(\gamma_{\a_{1}'}(x),\gamma_{\a_{2}'}(x))dm(x).
\]
Here and henceforth we write $a\rsa\a_{1},\a_{2}\rsa b$ to mean that
both $a\rsa\a_{1}\rsa b$ and $a\rsa\a_{2}\rsa b$.
\end{lem}

\begin{proof}
This is similar to arguments given by Jakobson and Naud in \cite[pgs. 466-467]{JNIJM}.
For $a\in\A$, let $\{\e_{k}^{a}\}_{k=1}^{\infty}$ be an orthonormal
basis for $\H(D_{a};\C)$ and let $\{v_{j}\}_{j=1}^{\dim V}$ be an
orthonormal basis for $V$. Then $\{\e_{k}^{a}\otimes v_{j}\::\:a\in\A,k\in\N,1\leq j\leq\dim V\,\}$
is an orthonormal basis for $\H(\D;V)$. We have 
\begin{align*}
\|\L_{\tau,s,\rho}\|_{\HS}^{2} & =\Tr(\L_{\tau,s,\rho}^{*}\L_{s,\tau,\rho})\\
 & =\sum_{a\in\A,k\in\N,1\leq j\leq\dim V}\langle\L_{\tau,s,\rho}[\e_{k}^{a}\otimes v_{j}],\L_{\tau,s,\rho}[\e_{k}^{a}\otimes v_{j}]\rangle\\
 & =\sum_{a\in\A,k\in\N,1\leq j\leq\dim V}\sum_{b\in\A}\int_{D_{b}}\langle\L_{\tau,s,\rho}[\e_{k}^{a}\otimes v_{j}](x),\L_{\tau,s,\rho}[\e_{k}^{a}\otimes v_{j}](x)\rangle dm(x)\\
 & =\sum_{a\in\A,k\in\N,1\leq j\leq\dim V}\sum_{b\in\A}\sum_{\substack{\a_{1},\a_{2}\in\ZZ(\tau)\\
\a_{1},\a_{2}\rsa b
}
}\\
 & \int_{D_{b}}\gamma'_{\a_{1}'}(x)^{s}\overline{\gamma'_{\a_{2}'}(x)^{s}}\langle\rho(\gamma_{\a_{1}'}^{-1})\e_{k}^{a}\otimes v_{j}(\gamma_{\a_{1}'}(x)),\rho(\gamma_{\a_{2}'}^{-1})\e_{k}^{a}\otimes v_{j}(\gamma_{\a_{2}'}(x))\rangle_{V}dm(x)\\
 & =\sum_{a,b\in\A,k\in\N}\sum_{\substack{\a_{1},\a_{2}\in\ZZ(\tau)\\
a\rsa\a_{1},\a_{2}\rsa b
}
}\\
 & \Tr(\rho(\gamma_{\a_{1}'}\gamma_{\a'_{2}}^{-1}))\int_{D_{b}}\gamma'_{\a_{1}'}(x)^{s}\overline{\gamma'_{\a_{2}'}(x)^{s}}\e_{k}^{a}(\gamma_{\a_{1}'}(x))\overline{\e_{k}^{a}(\gamma_{\a_{2}'}(x))}dm(x)\\
 & =\sum_{a,b\in\A}\sum_{\substack{\a_{1},\a_{2}\in\ZZ(\tau)\\
a\rsa\a_{1},\a_{2}\rsa b
}
}\Tr(\rho(\gamma_{\a_{1}'}\gamma_{\a'_{2}}^{-1}))\int_{D_{b}}\gamma'_{\a_{1}'}(x)^{s}\overline{\gamma'_{\a_{2}'}(x)^{s}}B_{D_{a}}(\gamma_{\a_{1}'}(x),\gamma_{\a_{2}'}(x))dm(x).
\end{align*}
The final application of Fubini's theorem is justified since we assume
$\tau\leq\tau_{0}$, so $\ZZ(\tau)\subset\W_{\geq2}$, and each $\gamma_{\a'_{1}},\gamma_{\a'_{2}}$
maps $D_{b}$ into a compact subset of $D_{a}$, coupled with the
fact that the convergence of $\sum_{k=1}^{\infty}\e_{k}^{\a}(x_{1})\overline{\e_{k}^{\a}(x_{2})}$
to $B_{D_{a}}(x_{1},x_{2})$ is uniform on compact subsets of $D_{a}$
(see, for example, \cite[Proof of Thm. 15.7]{Borthwick}). 
\end{proof}

\subsection{\label{subsec:A-pointwise-estimate}A pointwise estimate for the
modulus of a zeta function}
\begin{prop}[Pointwise bound for $|\zeta_{\tau,\rho}(s)|$]
\label{prop:pointiwse-bound}There is $\tau_{1}\leq\tau_{0}$ and
$B\in\R$ with $B>2\delta$ such that if $\tau\leq\tau_{1}$, if $s\in[B,\infty)$,
and $(\rho,V)$ is any finite dimensional unitary representation of
$\Gamma$, we have 
\[
-\log|\zeta_{\tau,\rho}(s)|\leq(\dim V)\tau.
\]
\end{prop}

\begin{rem}
A crucial restriction in Proposition \ref{prop:pointiwse-bound} is
$\Re(s)>2\delta$ that results from the presence of $\L_{\tau,s,\rho}^{2}$
in the definition of $\zeta_{\tau,\rho}$.
\end{rem}

\begin{proof}[Proof of Proposition \ref{prop:pointiwse-bound}]
We can write
\[
\zeta_{\tau,\rho}(s)=\det(1-\L_{\tau,s,\rho}^{2})=\exp\left(-\sum_{k=1}^{\infty}\frac{1}{k}\Tr\L_{\tau,s,\rho}^{2k}\right)
\]
whenever the series inside the exponential is absolutely convergent.
We have if $x\in D_{b}$
\[
\L_{\tau,s,\rho}^{2k}[f](x)=\sum_{\substack{\a_{1},\ldots,\a_{2k}\in\ZZ(\tau)\\
\a_{1}\rsa\a_{2}\rsa\cdots\rsa\a_{2k}\rsa b
}
}\gamma'_{\a_{1}'\a_{2}'\ldots\a'_{2k}}(x)^{s}\rho(\gamma_{\a_{1}'\a_{2}'\ldots\a'_{2k}}^{-1})f(\gamma_{\a_{1}'\a_{2}'\ldots\a'_{2k}}(x)).
\]
Carefully applying the Lefschetz fixed point formula \cite[Lemma 15.9]{Borthwick}
now gives
\[
\Tr\L_{\tau,s,\rho}^{2k}=\sum_{\substack{\a_{1},\ldots,\a_{2k}\in\ZZ(\tau)\\
\a_{2k}\rsa\a_{1}\rsa\a_{2}\rsa\cdots\rsa\a_{2k}
}
}\Tr[\rho(\gamma_{\a_{1}'\a_{2}'\ldots\a'_{2k}}^{-1})]\frac{\gamma'_{\a_{1}'\a_{2}'\ldots\a'_{2k}}(x_{\a_{1}'\a_{2}'\ldots\a'_{2k}})^{s}}{1-\gamma'_{\a_{1}'\a_{2}'\ldots\a'_{2k}}(x_{\a_{1}'\a_{2}'\ldots\a'_{2k}})}
\]
where $x_{\a_{1}'\a_{2}'\ldots\a'_{2k}}\in\R$ is the unique attracting
fixed point of $\gamma{}_{\a_{1}'\a_{2}'\ldots\a'_{2k}}$. Let $b$
denote the last letter of $\a_{2k}$.

By using Lemmas \ref{lem:deriv-to-upsilon} and \ref{lem:coarse-homomorphism}
($2k-1$ times) we obtain
\begin{align*}
\gamma'_{\a_{1}'\a_{2}'\ldots\a'_{2k}}(x_{\a_{1}'\a_{2}'\ldots\a'_{2k}}) & \leq K_{0}\Upsilon_{\a'_{1}\a'_{2}\cdots\a'_{2k-1}\a{}_{2k}}\\
 & \leq K_{0}K_{2}^{2k-1}\Upsilon_{\a{}_{1}}\ldots\Upsilon_{\a{}_{2k}}.
\end{align*}
Now using Lemma \ref{lem:upsilon-to-tau} we obtain
\[
\gamma'_{\a_{1}'\a_{2}'\ldots\a'_{2k}}(x_{\a_{1}'\a_{2}'\ldots\a'_{2k}})\leq K_{0}K_{1}^{2k-1}K_{2}^{2k-1}\tau^{2k}\leq K^{k}\tau^{2k}
\]
 for some $K>1$. We now assume 
\[
\tau_{1}\leq\frac{1}{2}K^{-1}
\]
so that given $\tau\leq\tau_{1}$ we have 
\[
\gamma'_{\a_{1}'\a_{2}'\ldots\a'_{2k}}(x_{\a_{1}'\a_{2}'\ldots\a'_{2k}})\leq2^{-2k}.
\]
We may also use the simple estimate $\Tr[\rho(\gamma_{\a_{1}'\a_{2}'\ldots\a'_{2k}}^{-1})]\leq\dim V$.
Putting this together gives
\[
|\Tr\L_{\tau,s,\rho}^{2k}|\leq(\dim V)\left(K\tau\right)^{2ks}|\ZZ(\tau)|^{2k}.
\]
Hence by Lemma \ref{lem:size-of-partition} we obtain
\[
|\Tr\L_{\tau,s,\rho}^{2k}|\leq(\dim V)\left(K\tau\right)^{2ks}C_{2}^{2k}\tau^{-2k\delta}=(\dim V)K^{2ks}C_{2}^{2k}\tau^{(2s-2\delta)k}.
\]
Choose $B$ such that $B>\max(1,2\delta)$ and 
\[
K^{B}\geq C_{2},
\]
with the effect of obtaining $|\Tr\L_{\tau,s,\rho}^{2k}|\leq(\dim V)K^{4ks}\tau^{(2s-2\delta)k}=(\dim V)(K^{4}\tau^{(2-\frac{2\delta}{s})})^{sk}$
when $s\geq B$. Now decrease $\tau_{1}$, if necessary, to ensure
\[
K^{4}\tau_{1}^{(1-\frac{2\delta}{B})}\leq2^{-1}.
\]
Note that $1-\frac{2\delta}{B}>0$, so this is indeed possible. The
result of our choices is that when $s\geq B\geq1$ and $\tau\leq\tau_{1}$
\[
\left|\det(1-\L_{\tau,s,\rho}^{2})\right|=\exp\left(\mathrm{Re}\left(-\sum_{k=1}^{\infty}\frac{1}{k}\Tr\L_{\tau,s,\rho}^{2k}\right)\right)\geq\exp\left(-(\dim V)\sum_{k=1}^{\infty}\left(\frac{\tau}{2}\right)^{sk}\right),
\]
so 
\[
-\log|\zeta_{\tau,\rho}(s)|\leq(\dim V)\sum_{k=1}^{\infty}\left(\frac{\tau}{2}\right)^{sk}\leq(\dim V)\sum_{k=1}^{\infty}\left(\frac{\tau}{2}\right)^{k}\leq(\dim V)\tau.
\]
\end{proof}

\section{The expectation of the Hilbert-Schmidt norm of the transfer operator\label{sec:The-expectation-of}}

\subsection{Statement of the main probabilistic estimate}

The main estimate we wish to prove in this Section \ref{sec:The-expectation-of}
is the following.
\begin{prop}
\label{prop:main-prob-estimate}Given $H_{1}>0$, $\sigma_{1}>\frac{3\delta}{4}$,
and $Q>\sigma_{1}$ there are constants $\epsilon=\epsilon(\Gamma,H_{1},Q,\sigma_{1})>0$,
and $n_{0}=n_{0}(\Gamma,H_{1},Q,\sigma_{1})>0$ such that if $\tau=n^{-\frac{2}{\delta}}$,
$n\geq n_{0}$, $s=\sigma+it$ with $\sigma\in[\sigma_{1},Q]$ and
$|t|\leq H_{1}$ we have
\[
\E_{n}\|\L_{\tau,s,\rho_{n}^{0}}\|_{\HS}^{2}\leq n^{-\epsilon}.
\]
\end{prop}

\subsection{The expected value of the trace of a word}

The key probabilistic estimate for $\rho_{n}^{0}$ that we use in
this paper is essentially due to Broder-Shamir \cite{BroderShamir},
and in the stronger form that we use it can be deduced from the work
of Puder \cite{PUDER}. We will explain how to deduce the result below.
\begin{thm}[Broder-Shamir, Puder]
\label{thm:broder-shamir-puder}Let $\gamma\in\Gamma$ have reduced
word length $t$. Then for any $n>t^{2}$ 
\[
\left|\E_{n}(\Tr[\rho_{n}^{0}(\gamma)])\right|\leq\begin{cases}
n-1 & \text{if \ensuremath{\gamma=\mathrm{id}},}\\
d(q)-1+\frac{t^{4}}{n-t^{2}} & \text{if \ensuremath{\gamma=\gamma_{0}^{q}}, \ensuremath{q\geq2} and \ensuremath{q} maximal,}\\
\frac{t^{4}}{n-t^{2}} & \text{otherwise}.
\end{cases}
\]
Here $d(q)$ is the number of divisors of $q$.
\end{thm}

\begin{rem}
Broder and Shamir \cite{BroderShamir} only prove upper bounds for
$\E_{n}(\Tr[\rho_{n}^{0}(\gamma)])$, whereas it is crucial for us
to have upper and lower bounds, since we deal with \emph{complex }weighted
sums of the random variables $\Tr[\rho_{n}^{0}(\gamma)]$.
\end{rem}

\begin{proof}[Deduction of Theorem \ref{thm:broder-shamir-puder}]
Let $\gamma$ be an element of the non-abelian free group $\Gamma$
with reduced word length $t$. Note that Theorem \ref{thm:broder-shamir-puder}
is trivial if $\gamma=\mathrm{id},$ so we assume this is not the
case. Puder proves in \cite[pg. 885]{PUDER} that for $n>t$ one has
an absolutely convergent Laurent series
\begin{equation}
\E_{n}(\Tr[\rho_{n}(\gamma)])=\sum_{S=0}^{\infty}\frac{a_{S}(\gamma)}{n^{S}}\label{eq:Laurent}
\end{equation}
where each $a_{S}(\gamma)\in\Z$. Puder associates to $\gamma$ a
quantity $\pi(\gamma)\in\N_{0}\cup\{\infty\}$ called the \emph{primitivity}\footnote{For good reasons, `primitivity' in the setting of \cite{PUDER} \uline{does
not} coincide with the notion of primitive closed geodesics, although
they are related. However, this is not relevant to the current proof.}\emph{ rank} of $\gamma$\emph{. }For our purposes, the only thing
we need to know is that $\pi(\gamma)=0$ if and only if $\gamma=\mathrm{id}$,
and $\pi(\gamma)=1$ if and only if $\gamma$ is a proper power. Puder
also considers a certain finite set $\Crit(\gamma)$ of subgroups
of the free group. Again, the only thing we need to know is that if
$\text{if \ensuremath{\gamma=\gamma_{0}^{q}}, \ensuremath{q\geq2} and \ensuremath{q} maximal,}$
then $|\Crit(\gamma)|=d(q)-1$ \cite[pg. 67]{PP15}. 

The following facts are proven by Puder in \cite[pp. 885-887]{PUDER}:
\begin{itemize}
\item We have $a_{0}(\gamma)=1$, unless $\pi(\gamma)=1$, in which case
\[
a_{0}(\gamma)=|\Crit(\gamma)|+1.
\]
\item If $1\leq S<\pi(\gamma)-1$ then 
\[
a_{S}(\gamma)=0.
\]
\item If $\pi(\gamma)\neq1$ then
\[
a_{\pi(\gamma)-1}=|\Crit(\gamma)|.
\]
\item For any $S\geq0$
\begin{align*}
|a_{S}(\gamma)| & \leq t^{2S+2}.
\end{align*}
\end{itemize}
Since $\Tr[\rho_{n}(\gamma)]=1+\Tr[\rho_{n}^{0}(\gamma)]$, $\text{if \ensuremath{\gamma=\gamma_{0}^{q}}, \ensuremath{q\geq2} and \ensuremath{q} maximal,}$
we have from (\ref{eq:Laurent})
\[
|\E_{n}(\Tr[\rho_{n}^{0}(\gamma)])|\leq d(q)-1+\sum_{S=1}^{\infty}\frac{t^{2S+2}}{n^{S}}=d(q)-1+\frac{t^{4}}{n-t^{2}}.
\]
If $\gamma$ is neither a proper power nor the identity then the estimate
is similar, but there is no $d(q)-1$ term since $\pi(\gamma)\geq2$.
\end{proof}

\subsection{Majorization of the expectation of the Hilbert-Schmidt norm}
\begin{lem}
\label{lem:maj-exp-hs-norm}Given $Q,H_{1}>0$ there is a constant
$C=C(\Gamma,H_{1},Q)$ such that if $\tau\leq\tau_{0}$ and $s=\sigma+it$
with $\sigma\in(0,Q]$ and $|t|\leq H_{1}$, 
\begin{equation}
\E_{n}\|\L_{\tau,s,\rho_{n}^{0}}\|_{\HS}^{2}\leq C\tau^{2\sigma}\sum_{\a,\b\in\ZZ(\tau)}|\E_{n}[\Tr(\rho_{n}^{0}(\gamma_{\a'}\gamma_{\b'}^{-1}))]|.\label{eq:maj-exp-hs-norm}
\end{equation}
\end{lem}

\begin{proof}
Suppose we are given $H_{1}$ as in the statement of the lemma. Taking
the expectation of the expression given in Lemma \ref{lem:expression-for-HS-norm}
gives
\begin{equation}
\E_{n}\|\L_{\tau,s,\rho_{n}^{0}}\|_{\HS}^{2}=\sum_{a,b\in\A}\sum_{\substack{\a_{1},\a_{2}\in\ZZ(\tau)\\
a\rsa\a_{1},\a_{2}\rsa b
}
}\E_{n}[\Tr(\rho_{n}^{0}(\gamma_{\a_{1}'}\gamma_{\a'_{2}}^{-1}))]\int_{D_{b}}\gamma'_{\a_{1}'}(x)^{s}\overline{\gamma'_{\a_{2}'}(x)^{s}}B_{D_{a}}(\gamma_{\a_{1}'}(x),\gamma_{\a_{2}'}(x))dm(x).\label{eq:hs-exp-formula}
\end{equation}
We wish to estimate the modulus of all quantities appearing in the
integral on the right hand side. Firstly the assumption that $\tau\leq\tau_{0}$
ensures $\ZZ(\tau)\subset\W_{\geq2}$, and so each $\gamma_{\a'_{1}},\gamma_{\a'_{2}}$
maps $D_{b}$ into a compact subset of $D_{a}$. It then follows from
the explicit expression for the Bergman kernel in (\ref{eq:explicit-bergman})
that there is $K=K(\Gamma)>0$ such that 
\begin{equation}
B_{D_{a}}(\gamma_{\a_{1}'}(x),\gamma_{\a_{2}'}(x))\leq K\label{eq:bergman-maj}
\end{equation}
for all $a,x,\a_{1},\a_{2}$ as in (\ref{eq:hs-exp-formula}).

By definition, if $s=\sigma+it$,
\begin{align*}
(\gamma'_{\a'_{1}}(x))^{s} & =\exp\left((\sigma+it)(\log|\gamma_{\a'_{1}}'(x)|+i\arg(\gamma'_{\a'_{1}}(x))\right)
\end{align*}
where $\arg$ is the principal value of the argument, $\arg:\C-\R_{\leq0}\to(-\pi,\pi)$.
Hence
\[
|(\gamma'_{\a'_{1}}(x))^{s}|=\exp(\sigma\log|\gamma'_{\a'_{1}}(x)|-t\arg(\gamma'_{\a'_{1}}(x))\leq e^{\pi|t|}|\gamma'_{\a'_{1}}(x)|^{\sigma}.
\]
Therefore by Lemma \ref{lem:deriv-to-tau} for some $c=c(H_{1},Q)>0$
we have for $|t|\leq H_{1}$
\begin{equation}
|(\gamma'_{\a'_{1}}(x))^{s}|\leq c\tau^{\sigma}.\label{eq:deriv-maj}
\end{equation}
for all $\a'_{1}$, $x$ in (\ref{eq:hs-exp-formula}), and the same
for $\a_{2}$ in place of $\a_{1}$. Hence applying the triangle inequality
to (\ref{eq:hs-exp-formula}) and using (\ref{eq:bergman-maj}) and
(\ref{eq:deriv-maj}), together with the fact that the $D_{b}$ have
finite Lebesgue measure gives
\begin{align*}
\E_{n}\|\L_{\tau,s,\rho_{n}^{0}}\|_{\HS}^{2} & \leq C_{0}\tau^{2\sigma}\sum_{a,b\in\A}\sum_{\substack{\a_{1},\a_{2}\in\ZZ(\tau)\\
a\rsa\a_{1},\a_{2}\rsa b
}
}|\E_{n}[\Tr(\rho_{n}^{0}(\gamma_{\a_{1}'}\gamma_{\a'_{2}}^{-1}))]|\\
 & \leq C\tau^{2\sigma}\sum_{\a,\b\in\ZZ(\tau)}|\E_{n}[\Tr(\rho_{n}^{0}(\gamma_{\a'}\gamma_{\b'}^{-1}))]|
\end{align*}
for some $C=C(\Gamma,H_{1},Q)$ whenever $|t|\leq H_{1}$ and $\tau\leq\tau_{0}$. 
\end{proof}
The next step is to input the estimates of Theorem \ref{thm:broder-shamir-puder}
into the estimate of Lemma \ref{lem:maj-exp-hs-norm}. To organize
the result we introduce, for each $q\in\Z_{\geq2}$, the set 
\begin{align*}
\Pairs(\tau;q) & \eqdf\{(\a,\b)\in\ZZ(\tau)\times\ZZ(\tau),\:\gamma_{\a'}\gamma_{\b'}^{-1}\text{ is a \ensuremath{q^{th}} power in \ensuremath{\Gamma} with }\,q\text{ maximal}\,\},
\end{align*}
and 
\[
\Pairs(\tau)\eqdf\bigcup_{q\geq2}\Pairs(\tau;q).
\]
Notice that in the above, $\gamma_{\a'}\gamma_{\b'}^{-1}\neq\mathrm{id}$.
We will show
\begin{lem}
\label{lem:splitting-into-three-sums}Given $Q,H_{1},\alpha,\epsilon>0$,
there are constants $C=C(\Gamma,H_{1},Q)>0$ and $n_{0}=n_{0}(\Gamma,\epsilon,\alpha)$
such that if $\tau=n^{-\alpha}$ and $s=\sigma+it$ with $\sigma\in(0,Q]$,
$|t|\leq H_{1}$, and $n\geq n_{0}$, we have 
\[
\E_{n}\|\L_{\tau,s,\rho_{n}^{0}}\|_{\HS}^{2}\leq C\tau^{2\sigma}\left(n\tau^{-\delta}+n^{\epsilon}|\Pairs(\tau)|+\frac{1}{n^{1-\epsilon}}\tau^{-2\delta}\right).
\]
\end{lem}

\begin{proof}
We will input Theorem \ref{thm:broder-shamir-puder} into Lemma \ref{lem:maj-exp-hs-norm}.
For this to be valid we need to control the word lengths of elements
of $\ZZ(\tau)$. By Lemma \ref{lem:tau-to-word-length}, all $\a\in\ZZ(\tau)$
have $|\a|\leq c\log\tau^{-1}+\kappa$, so if $\tau=n^{-\alpha}$
with $\alpha>0$, 
\[
|\a|\leq c\alpha\log n+\kappa<\frac{1}{2}n^{\frac{1}{2}}
\]
for $n$ sufficiently large, say $n\geq n_{0}.$ In this case, if
$\a,\b\in\ZZ(\tau)$ the reduced word length of $\gamma_{\a'}\gamma_{\b'}^{-1}$
is 
\[
<n^{\frac{1}{2}}
\]
so we may apply Theorem \ref{thm:broder-shamir-puder} to $\E_{n}[\Tr(\rho_{n}^{0}(\gamma_{\a'}\gamma_{\b'}^{-1}))]$.
Moreover, if $t$ is the reduced word length of $\gamma_{\a'}\gamma_{\b'}^{-1}$,
we have $t\leq2c\alpha\log n+\kappa$ so for any $\epsilon>0$, we
have 
\[
\frac{t^{2}}{n-t^{2}}\leq\frac{1}{n^{1-\epsilon}}
\]
when $n\geq n_{0}$, after increasing $n_{0}$ if necessary. Finally,
in the case $\gamma_{\a'}\gamma_{\b'}^{-1}$ is a $q^{th}$ power
in the free group $\Gamma$, with $q\geq2$ we must have $q\leq t$
and so $d(q)\leq t\leq2c\alpha\log n+\kappa\leq n^{\epsilon}$ for
any $\epsilon>0$ and $n\geq n_{0}(\epsilon)$ (here we increase $n_{0}$
again if necessary). 

With these estimates in hand, we partition the range of the sum of
the right hand of (\ref{eq:maj-exp-hs-norm}) according to the following
three cases:
\begin{itemize}
\item $\gamma_{\a'}\gamma_{\b'}^{-1}$ is the identity; if this is the case
then $|\E_{n}[\Tr(\rho_{n}^{0}(\gamma_{\a'}\gamma_{\b'}^{-1}))]|=n-1\leq n$.
We observe that $\gamma_{\a'}\gamma_{\b'}^{-1}=\mathrm{id}$ implies
$\gamma_{\a'}=\gamma_{\b'}$, but since the map $\a\to\gamma_{\a}$
is one-to-one, this forces $\a'=\b'$. Therefore the number of pairs
$(\a,\b)$ of this type is $\leq|\A||\bar{Z}(\tau)|\leq|\A|C_{2}\tau^{-\delta}$
by Lemma \ref{lem:size-of-partition}. So in total, these pairs contribute
at most
\begin{equation}
C|\A|C_{2}\tau^{2\sigma}n\tau^{-\delta}\label{eq:identity-cont}
\end{equation}
 to the bound for $\E_{n}\|\L_{\tau,s,\rho_{n}^{0}}\|_{\HS}^{2}$
given in (\ref{eq:maj-exp-hs-norm}).
\item $\gamma_{\a'}\gamma_{\b'}^{-1}$ is a $q$th power with $q$ maximal,
$q\geq2$. In this case, Theorem \ref{thm:broder-shamir-puder} gives
\[
|\E_{n}[\Tr(\rho_{n}^{0}(\gamma_{\a'}\gamma_{\b'}^{-1}))]|\leq d(q)-1+\frac{1}{n^{1-\epsilon}}\leq2n^{\epsilon}
\]
for $n\geq n_{0}$. The total number of these pairs (for all possible
$q$) is $|\Pairs(\tau)|$ so in total, these pairs contribute at
most
\begin{equation}
2C\tau^{2\sigma}n^{\epsilon}|\Pairs(\tau)|\label{eq:powers-cont}
\end{equation}
to (\ref{eq:maj-exp-hs-norm}).
\item If $\gamma_{\a'}\gamma_{\b'}^{-1}$ is not the identity and not a
proper power, then Theorem \ref{thm:broder-shamir-puder} gives
\[
|\E_{n}[\Tr(\rho_{n}^{0}(\gamma_{\a'}\gamma_{\b'}^{-1}))]|\leq\frac{1}{n^{1-\epsilon}}.
\]
We overestimate how many pairs of this kind there are by counting
all pairs, of which there are $|\bar{Z}(\tau)|^{2}\leq C_{2}^{2}\tau^{-2\delta}$
by Lemma \ref{lem:size-of-partition}. So in total, these pairs contribute
at most
\begin{equation}
CC_{2}^{2}\tau^{2\sigma}\frac{\tau^{-2\delta}}{n^{1-\epsilon}}\label{eq:non-power-cont}
\end{equation}
to (\ref{eq:maj-exp-hs-norm}).
\end{itemize}
Summing up the bounds (\ref{eq:identity-cont}), (\ref{eq:powers-cont}),
and (\ref{eq:non-power-cont}) gives the result.
\end{proof}
In the next section, we will estimate $|\Pairs(\tau)|$.

\subsection{\label{subsec:Estimating-sigma2}Estimating the size of $\text{\ensuremath{\protect\Pairs}(\ensuremath{\tau})}$}

Our goal is now to prove the following proposition controlling the
size of $\Pairs(\tau)$.
\begin{prop}
\label{prop:key-prop}For any $\epsilon>0$, there is $\tau_{2}=\tau_{2}(\Gamma,\epsilon)$
such that for $\tau\leq\tau_{2}$
\[
|\Pairs(\tau)|\leq\tau^{-\delta-\epsilon}.
\]
\end{prop}

In the remainder of this $\S$\ref{subsec:Estimating-sigma2} we prove
Proposition \ref{prop:key-prop}. 

We decompose $\Pairs(\tau)$ as follows. We introduce integer parameters
$L,R\geq0$, $M_{1},M_{2}\geq0$, and $q\geq2$. For such parameters,
let $\Pairs(\text{\ensuremath{\tau,}}L,M_{1},M_{2},R;q)$ be the subset
of $\Pairs(\tau;q)$ consisting of those $(\a,\b)\in\Pairs(\tau;q)$
with
\begin{align*}
|\a'|=N_{1} & \eqdf L+M_{1}+R,\\
|\b'|=N_{2} & \eqdf L+M_{2}+R,
\end{align*}
\[
\a'=(a_{1},\ldots,a_{N_{1}}),\quad\b'=(b_{1},\ldots,b_{N_{2}}),
\]
\[
a_{1}=b_{1},a_{2}=b_{2},\ldots,a_{L}=b_{L},a_{L+1}\neq b_{L+1},
\]
\[
a_{N_{1}}=b_{N_{2}},a_{N_{1}-1}=b_{N_{2}-1},\ldots,a_{N_{1}-R+1}=b_{N_{2}-R+1},a_{N_{1}-R}\neq b_{N_{2}-R}.
\]
Since every element of $\Pairs(\tau)$ belongs to some $\Pairs(\text{\ensuremath{\tau,}}L,M_{1},M_{2},R;q)$,
we have
\begin{equation}
|\Pairs(\tau)|\leq\sum_{L,M_{1},M_{2},R\geq0,q\geq2}|\Pairs(\text{\ensuremath{\tau,}}L,M_{1},M_{2},R;q)|.\label{eq:splitting_up_sigma_2}
\end{equation}

We will estimate $|\Pairs(\text{\ensuremath{\tau,}}L,M_{1},M_{2},R;q)|$
in the following lemma.
\begin{lem}
\label{lem:sigma3-partitioned}There are constants $D>0$ and $\kappa\in\R$
depending only on $\Gamma$ such that\\
 $\Pairs(\text{\ensuremath{\tau,}}L,M_{1},M_{2},R;q)$ is empty unless
\begin{equation}
D^{-1}\log\tau^{-1}-\kappa\leq N_{1},\,N_{2}\leq D\log\tau^{-1}+\kappa,\label{eq:range-of-N_1-N_2}
\end{equation}
and
\begin{equation}
2\leq q\leq D\log\tau^{-1}+\kappa.\label{eq:q-bounds}
\end{equation}
Under the same assumptions as Proposition \ref{prop:key-prop}, and
assuming (\ref{eq:range-of-N_1-N_2}) holds, there is a constant $K=K(\Gamma)>0$
such that
\begin{equation}
|\Pairs(\text{\ensuremath{\tau,}}L,M_{1},M_{2},R;q)|\leq K\tau^{-\delta}.\label{eq:sigma2-bound}
\end{equation}
\end{lem}

\begin{proof}
For the first statement of the lemma, if $(\a,\b)\in\Pairs(\text{\ensuremath{\tau,}}L,M_{1},M_{2},R;q)$,
since $|\a'|=N_{1}$, $|\b'|=N_{2}$, and $\a,\b\in\ZZ(\tau)$, Lemma
\ref{lem:tau-to-word-length} implies (\ref{eq:range-of-N_1-N_2})
must hold; therefore\\
$\Pairs(\text{\ensuremath{\tau,}}L,M_{1},M_{2},R;q)$ is empty if
(\ref{eq:range-of-N_1-N_2}) does not hold. Moreover since $2\leq q\leq N_{1}+N_{2}$,
after doubling $D$ and $\kappa$, (\ref{eq:q-bounds}) must hold
also.

Now we prove (\ref{eq:sigma2-bound}) assuming (\ref{eq:range-of-N_1-N_2})
holds. For $(\a,\b)\in\Pairs(\text{\ensuremath{\tau,}}L,M_{1},M_{2},R;q)$
we have 
\[
\a'=\c\AA\d,\quad\b'=\c\BB\d
\]
where $|\c|=L$, $|\d|=R$, $|\AA|=M_{1}$ and $|\BB|=M_{2}$. Therefore
\[
\gamma_{\a'}\gamma_{\b'}^{-1}=\gamma_{\c}\gamma_{\AA}\gamma_{\d}\gamma_{\d}^{-1}\gamma_{\BB}^{-1}\gamma_{\c}^{-1}=\gamma_{\c}\gamma_{\AA}\gamma_{\BB}^{-1}\gamma_{\c}^{-1}.
\]
Since $\gamma_{\a'}\gamma_{\b'}^{-1}$ is a $q$th power, with $q$
maximal, $\gamma_{\AA}\gamma_{\BB}^{-1}$ is also a $q$th power,
with $q$ maximal, as they are conjugate in $\Gamma$. Since the first
letters of $\AA$ and $\BB$ are not the same, and the last letters
of $\AA$ and $\BB$ are not the same, we have $\AA\to\overline{\BB}\to\AA$,
in other words, the word $\AA\overline{\BB}$ is cyclically reduced.
It now follows that there is some $\mathbf{u}\in\W^{\circ}$ with
$|\u|=\frac{M_{1}+M_{2}}{q}$ and $\u\to\u$ (i.e. $\u$ is also cyclically
reduced) such that
\[
\AA\overline{\BB}=\underbrace{\u\u\ldots\u}_{q};
\]
$\AA\overline{\BB}$ is $q$ repeated copies of $\u$. Therefore
\begin{equation}
\AA=\underbrace{\u\u\ldots\u}_{q_{1}}\v_{1},\quad\BB=\underbrace{\bar{\u}\bar{\u}\ldots\bar{\u}}_{q_{2}}\v_{2}\label{eq:alpha-beta-expressions}
\end{equation}
with $\v_{1}\rs\overline{\v_{2}}$ and
\begin{align}
\v_{1}\overline{\v_{2}} & =\u,\label{eq:v1v2-relation}
\end{align}
\[
q_{1}+q_{2}=q-1.
\]
\uline{Our estimates will crucially rely on the observation that
for fixed \mbox{$L,M_{1},M_{2},R,q$}, choosing \mbox{$\c,\d,\u$}
specifies \mbox{$\a'$} and \mbox{$\b'$} and hence specifies \mbox{$\a$}
and \mbox{$\b$} except for their last letters.}

We will use the shorthand $\u^{m}\eqdf\underbrace{\u\u\ldots\u}_{m}$
for $m\in\N$. From (\ref{eq:alpha-beta-expressions}), using Lemma
\ref{lem:coarse-homomorphism} three times gives
\begin{align*}
\Upsilon_{\a'} & \leq K_{2}^{3}\Upsilon_{\c}\Upsilon_{\u^{q_{1}}}\Upsilon_{\v_{1}}\Upsilon_{\d}
\end{align*}
and using the same estimate in addition to the mirror estimate of
Lemma \ref{lem:mirror-1} gives
\begin{align*}
\Upsilon_{\b'} & \leq K_{2}^{3}\Upsilon_{\c}\Upsilon_{\overline{\u}^{q_{2}}}\Upsilon_{\overline{\v_{2}}}\Upsilon_{\d}\\
 & \leq K_{2}^{3}K_{3}\Upsilon_{\c}\Upsilon_{\u^{q_{2}}}\Upsilon_{\overline{\v_{2}}}\Upsilon_{\d}
\end{align*}
Therefore, now using Lemma \ref{lem:coarse-homomorphism} in the opposite
direction together with (\ref{eq:v1v2-relation}) we obtain
\begin{align}
\Upsilon_{\a'}\Upsilon_{\b'} & \leq K_{2}^{6}K_{3}\Upsilon_{\c}^{2}\Upsilon_{\u^{q_{1}}}\Upsilon_{\u^{q_{2}}}\Upsilon_{\v_{1}}\Upsilon_{\overline{\v_{2}}}\Upsilon_{\d}^{2}\nonumber \\
 & \leq K_{2}^{7}K_{3}\Upsilon_{\c}^{2}\Upsilon_{\u^{q_{1}}}\Upsilon_{\u^{q_{2}}}\Upsilon_{\u}\Upsilon_{\d}^{2}\nonumber \\
 & \leq K_{2}^{9}K_{3}\Upsilon_{\c}^{2}\Upsilon_{\u^{q}}\Upsilon_{\d}^{2}.\label{eq:upsilon-product-bound}
\end{align}
To exploit these estimates, we note that by Lemma \ref{lem:upsilon-to-tau},
we have
\begin{align*}
|\Pairs(\text{\ensuremath{\tau,}}L,M_{1},M_{2},R;q)| & \leq K_{1}^{\frac{\delta}{2}}\tau^{-\delta}\sum_{(\a,\b)\in\Pairs(\text{\ensuremath{\tau,}}L,M_{1},M_{2},R;q)}\Upsilon_{\a}^{\frac{\delta}{2}}\Upsilon_{\b}^{\frac{\delta}{2}}\\
 & \leq K_{1}^{\frac{\delta}{2}}\tau^{-\delta}\sum_{(\a,\b)\in\Pairs(\text{\ensuremath{\tau,}}L,M_{1},M_{2},R;q)}\Upsilon_{\a'}^{\frac{\delta}{2}}\Upsilon_{\b'}^{\frac{\delta}{2}}
\end{align*}
where the last inequality is by (\ref{eq:trivial-child-ineq}). Let
\[
\Sigma=\Sigma(\text{\ensuremath{\tau,}}L,M_{1},M_{2},R,q)\eqdf\sum_{(\a,\b)\in\Pairs(\text{\ensuremath{\tau,}}L,M_{1},M_{2},R;q)}\Upsilon_{\a'}^{\frac{\delta}{2}}\Upsilon_{\b'}^{\frac{\delta}{2}},
\]
so we have 
\begin{equation}
|\Pairs(\text{\ensuremath{\tau,}}L,M_{1},M_{2},R;q)|\leq K_{1}^{\frac{\delta}{2}}\tau^{-\delta}\Sigma.\label{eq:power-pairs-tosigma}
\end{equation}
By (\ref{eq:upsilon-product-bound})
\begin{align}
\Sigma & =\sum_{(\a,\b)\in\Pairs(\text{\ensuremath{\tau,}}L,M_{1},M_{2},R;q)}\Upsilon_{\a'}^{\frac{\delta}{2}}\Upsilon_{\b'}^{\frac{\delta}{2}}\nonumber \\
 & \leq|\A|^{2}\sum_{\c,\u,\d}(K_{2}^{9}K_{3})^{\frac{\delta}{2}}\Upsilon_{\c}^{\delta}\Upsilon_{\u^{q}}^{\frac{\delta}{2}}\Upsilon_{\d}^{\delta}\nonumber \\
 & \leq|\A|^{2}(K_{2}^{9}K_{3})^{\frac{\delta}{2}}\left(\sum_{\c\in\W_{L}}\Upsilon_{\c}^{\delta}\right)\left(\sum_{\u\in\W_{(M_{1}+M_{2})/q}}\Upsilon_{\u^{q}}^{\frac{\delta}{2}}\right)\left(\sum_{\d\in\W_{R}}\Upsilon_{\d}^{\delta}\right).\label{eq:Sigma3-partitioned-first-estimate}
\end{align}
By Lemma \ref{lem:Pressure-estimate} there is some $C=C(\Gamma)>0$
such that
\begin{equation}
\sum_{\c\in\W_{L}}\Upsilon_{\c}^{\delta}\leq C\exp(LP(\delta))=C,\quad\sum_{\d\in\W_{R}}\Upsilon_{\d}^{\delta}\leq C\exp(RP(\delta))=C.\label{eq:end-terms}
\end{equation}
To deal with $\sum_{\u\in\W_{(M_{1}+M_{2})/q}}\Upsilon_{\u^{q}}^{\frac{\delta}{2}}$,
we write
\[
q=2\tilde{q}+r
\]
where $r=1$ if $q$ is odd and $r=0$ if $q$ is even. Now using
Lemma \ref{lem:coarse-homomorphism} twice and the uniform bound for
$\Upsilon_{\u}$ from (\ref{eq:upsilon-bounded}) we obtain
\begin{align*}
\Upsilon_{\u} & \leq K_{2}^{2}\Upsilon_{\u^{\tilde{q}}}\Upsilon_{\u^{\tilde{q}}}\Upsilon_{\u^{r}}\leq cK_{2}^{2}\Upsilon_{\u^{\tilde{q}}}^{2}.
\end{align*}
Therefore
\begin{align}
\sum_{\u\in\W_{(M_{1}+M_{2})/q}}\Upsilon_{\u^{q}}^{\frac{\delta}{2}} & \leq(cK_{2}^{2})^{\frac{\delta}{2}}\sum_{\u\in\W_{(M_{1}+M_{2})/q}}\Upsilon_{\u^{\tilde{q}}}^{\delta}\nonumber \\
 & \leq(cK_{2}^{2})^{\frac{\delta}{2}}\sum_{\mathbf{U}\in\W_{\tilde{q}(M_{1}+M_{2})/q}}\Upsilon_{\mathbf{U}}^{\delta}\nonumber \\
 & \leq C(cK_{2}^{2})^{\frac{\delta}{2}}\exp\left(\frac{\tilde{q}(M_{1}+M_{2})}{q}P(\delta)\right)\nonumber \\
 & =C(cK_{2}^{2})^{\frac{\delta}{2}}.\label{eq:middle-term}
\end{align}
where the final inequality is by Lemma \ref{lem:Pressure-estimate}
and $C=C(\delta)$ is the constant provided there. Therefore in total,
inputting our bounds (\ref{eq:end-terms}) and (\ref{eq:middle-term})
into (\ref{eq:Sigma3-partitioned-first-estimate}) we get
\[
\Sigma\leq\tilde{K}
\]
 for $\tilde{K}=\tilde{K}(\Gamma)>0$. Hence by (\ref{eq:power-pairs-tosigma})
\[
|\Pairs(\text{\ensuremath{\tau,}}L,M_{1},M_{2},R;q)|\leq K\tau^{-\delta}
\]
for some $K=K(\Gamma)>0$.
\end{proof}
Now we can prove Proposition \ref{prop:key-prop}.
\begin{proof}[Proof of Proposition \ref{prop:key-prop}]
Combining (\ref{eq:splitting_up_sigma_2}) with Lemma \ref{lem:sigma3-partitioned}
we obtain for constants $D,\kappa>0$
\begin{align*}
|\Pairs(\tau)| & \leq\sum_{0\leq L,M_{1},M_{2},R,q\leq D\log\tau^{-1}+\kappa}K\tau^{-\delta}\\
 & \leq K\left(D\log\tau^{-1}+\kappa\right)^{5}\tau^{-\delta}\\
 & \leq\tau^{-\delta-\epsilon}
\end{align*}
for any $\epsilon>0$ and $\tau$ sufficiently small.
\end{proof}

\subsection{Proof of Proposition \ref{prop:main-prob-estimate}}

Suppose we are given parameters $Q,H_{1},\sigma_{1}$ as in Proposition
\ref{prop:main-prob-estimate}. We assume $\sigma_{1}>\frac{3\delta}{4}$,
$\sigma\in[\sigma_{1},Q]$, and $|t|\leq H_{1}$. Let 
\begin{align*}
\epsilon & \eqdf\min\left(\frac{\delta}{4},\frac{1}{4}\left(\frac{4\sigma_{1}}{\delta}-3\right)\right)>0.
\end{align*}
Let $\alpha=\frac{2}{\delta}$ and let $\tau=n^{-\alpha}=n^{-\frac{2}{\delta}}$.
Let $n_{0}$ be such that for $n\geq n_{0}$, $\tau\leq\tau_{2}$
where $\tau_{2}$ is the one provided by Proposition \ref{prop:key-prop}
for the current $\epsilon$, and $n_{0}$ is at least the one provided
by Lemma \ref{lem:splitting-into-three-sums} for the current $\epsilon$
and $\alpha$.

Combining Proposition \ref{prop:key-prop} and Lemma \ref{lem:splitting-into-three-sums}
gives for $n\geq n_{0}$
\begin{align*}
\E_{n}\|\L_{\tau,s,\rho_{n}^{0}}\|_{\HS}^{2} & \leq C\tau^{2\sigma}\left(n\tau^{-\delta}+n^{\epsilon}\tau^{-\delta-\epsilon}+\frac{1}{n^{1-\epsilon}}\tau^{-2\delta}\right)\\
 & \leq2C\left(n^{1-(2\sigma-\delta)\alpha}+n^{-1+\epsilon-2(\sigma-\delta)\alpha}\right)\\
 & \leq4Cn^{\left(3-\frac{4\sigma}{\delta}\right)+\epsilon}\\
 & \leq4Cn^{\left(3-\frac{4\sigma_{1}}{\delta}\right)+\epsilon}\leq4Cn^{-3\epsilon}\leq n^{-\epsilon}
\end{align*}
after possibly increasing $n_{0}$. This completes the proof of Proposition
\ref{prop:main-prob-estimate}. $\square$

\section{Proof of Theorems \ref{thm:L2-gap} and \ref{thm:rectangle-gap}}

As explained in the Introduction, Theorem \ref{thm:rectangle-gap}
implies Theorem \ref{thm:L2-gap}, so we will prove Theorem \ref{thm:rectangle-gap}.
A direct proof of Theorem \ref{thm:L2-gap} would use most of the
same ideas and not be significantly shorter.

Let $n\in\N$, $\phi_{n}$ be a random homomorphism $\phi_{n}:\Gamma\to S_{n}$,
and $(\rho_{n}^{0},V_{n}^{0})$ be the random representation described
in $\S$\ref{subsec:The-representations-appearing}. Let $X_{n}$
be the random convex co-compact hyperbolic surface described in the
Introduction. 

Let $\sigma_{0}\in(\frac{3}{4}\delta,\delta)$ and $H$ be the number
given in the assumptions of Theorem \ref{thm:rectangle-gap}. Let
$\tau_{1}$ and $B>2\delta$ be the constants provided by Proposition
\ref{prop:pointiwse-bound} and choose $b\geq B$ such that the open
disc $D_{b-\frac{3}{4}\delta}(b)$ contains $\Rect(\sigma_{0},H)$.
We let $\tau=n^{-\frac{2}{\delta}}$.

Since as $n$ varies in $\N$ and $\phi_{n}$ runs over all homomorphisms
from $\Gamma\to S_{n}$, the countable collection of holomorphic functions
$\zeta_{\tau,\rho_{n}^{0}}$ have amongst them all, a countable number
of zeros in the closed disc $\overline{D_{b-\frac{3}{4}\delta}(b)}$,
it is possible to find a $\sigma_{1}\in(\frac{3}{4}\delta,\sigma_{0})$
such that
\begin{itemize}
\item no $\zeta_{\tau,\rho_{n}^{0}}$ has a zero $s$ with $|s-b|=b-\sigma_{1},$
and
\item the open disc $D_{b-\sigma_{1}}(b)$ contains the closed rectangle
$\Rect(\sigma_{0},H)$
\end{itemize}
\begin{figure}
\includegraphics{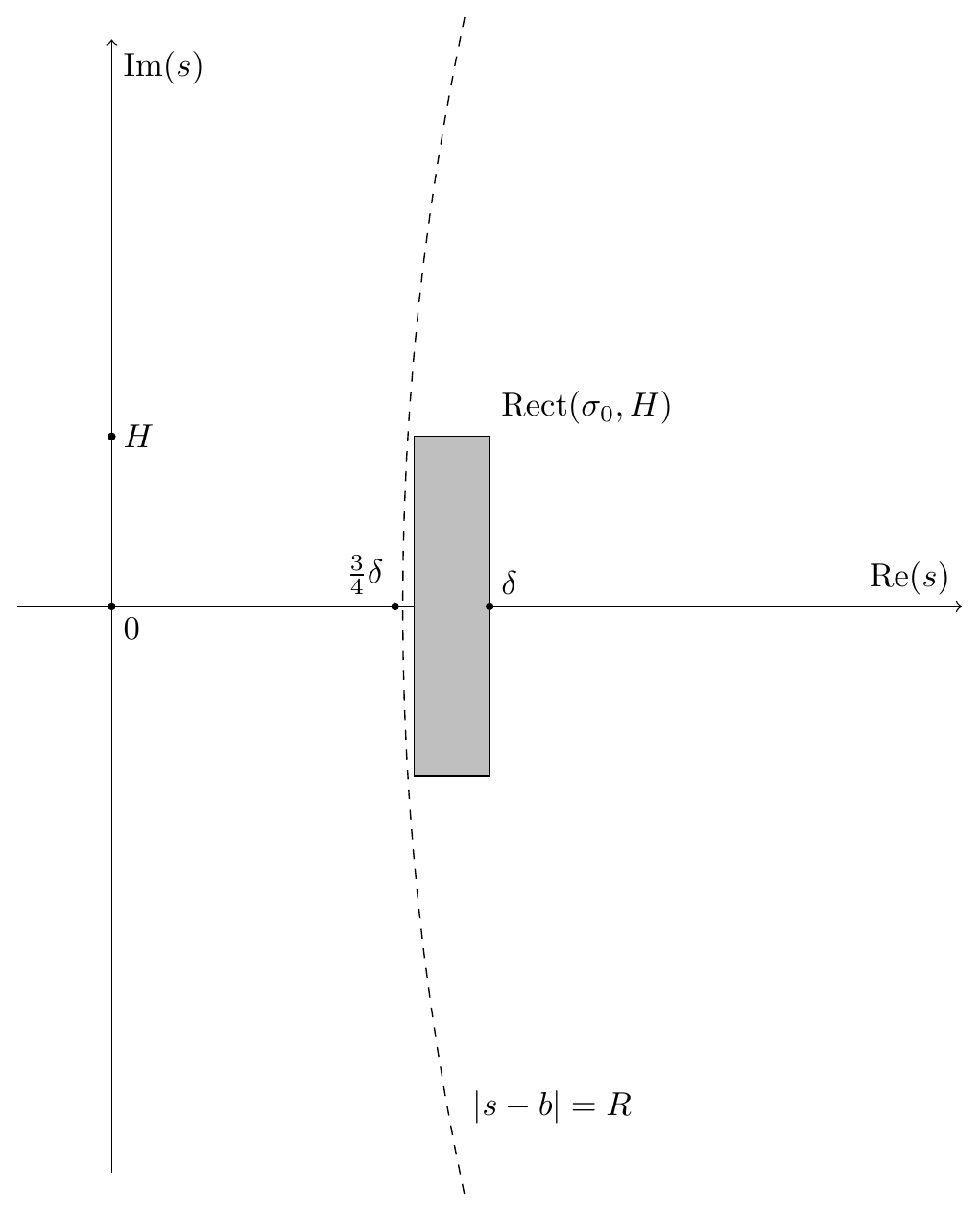}

\caption{Illustration of the contour used in Jensen's formula\label{fig:jensen}.}

\end{figure}
We pick such a $\sigma_{1}.$ Now we let
\[
R\eqdf b-\sigma_{1},\quad R'\eqdf\sup_{s\in\Rect(\sigma_{0},h)}|b-s|<R.
\]
We will shortly apply Proposition \ref{prop:main-prob-estimate} with
$Q=b+R$, $\sigma_{1}$ as is it is in the current context, and $H_{1}=R$.
Let $\epsilon$ and $n_{0}$ be the positive constants provided by
these inputs to Proposition \ref{prop:main-prob-estimate}. We pick
$n_{1}\geq n_{0}$ such that for $n\geq n_{1}$, $\tau\leq\tau_{1}$.
This sets up all the constants for the proof.

If for $\sigma>0$, $\sigma$ is an resonance for $X_{n}$, and is
either not a resonance of $X$ or a resonance of $X$ with a lower
multiplicity, then by Theorem \ref{thm:Patterson-Perry} combined
with Corollary \ref{cor:new-zeros-zeros-of-zeta}, $\zeta_{\tau,\rho_{n}^{0}}(\sigma)=0$.
\emph{Therefore it suffices to show that a.a.s. there are no zeros
of $\zeta_{\tau,\rho_{n}^{0}}$ in $\Rect(\sigma_{0},H)$. }

Let $\mathcal{N}(\phi_{n})$ be the number of zeros of $\zeta_{\tau,\rho_{n}^{0}}$
in $\Rect(\sigma_{0},H)$. Note that $\Rect(\sigma_{0},H)\subset\overline{D_{R'}(b)}$.
By Jensen's formula \cite[Thm. A.2]{Borthwick} applied to the translate
of $\zeta_{\tau,\rho_{n}^{0}}$ by $b$ we have 
\begin{equation}
\sum_{\substack{z\in D_{R}(b)\\
\zeta_{\tau,\rho_{n}^{0}}(z)=0
}
}^{*}\log\left(\frac{R}{|z-b|}\right)=\mathcal{M}(\phi_{n})\eqdf\frac{1}{2\pi}\int_{0}^{2\pi}\log|\zeta_{\tau,\rho_{n}^{0}}(b+Re^{i\theta})|d\theta-\log|\zeta_{\tau,\rho_{n}^{0}}(b)|.\label{eq:jensen's formula}
\end{equation}
The star on the sum means zeros are repeated according to their multiplicity.
Note that $b\geq B$ so Proposition \ref{prop:pointiwse-bound} ensures
$\zeta_{\tau,\rho_{n}^{0}}(b)\neq0$, and the choice of $\sigma_{1}$
ensures $\zeta_{\tau,\rho_{n}^{0}}(b+Re^{i\theta})$ is never zero.
These conditions were needed for Jensen's formula. Now (\ref{eq:jensen's formula})
implies
\begin{align}
\mathcal{N}(\phi_{n}) & \leq\log\left(\frac{R}{R'}\right)^{-1}\mathcal{M}(\phi_{n}).\label{eq:NtoM}
\end{align}
Next we majorize $\mathcal{M}(\phi_{n})$. By Weyl's inequality (cf.
\cite[(A36)]{Borthwick}) we have for any $s\in\C$
\begin{equation}
\log|\zeta_{\tau,\rho_{n}^{0}}(s)|=\log|\det(1-\L_{\tau,s,\rho_{n}^{0}}^{2})|\leq\|\L_{\tau,s,\rho_{n}^{0}}^{2}\|_{1}\leq\|\L_{\tau,s,\rho_{n}^{0}}\|_{\HS}^{2},\label{eq:first-est}
\end{equation}
where $\|\bullet\|_{1}$ and $\|\bullet\|_{\HS}$ stand for the trace
and Hilbert-Schmidt norms, respectively. This was the reason for the
square in the definition of $\zeta_{\tau,\rho_{n}^{0}}(s)$. Also,
by Proposition \ref{prop:pointiwse-bound} we have
\begin{equation}
-\log|\zeta_{\tau,\rho_{n}^{0}}(b)|\leq(n-1)\tau\leq n^{1-\frac{2}{\delta}}\leq n^{-1}\label{eq:second-est}
\end{equation}
since $\delta\in(0,1)$. Using (\ref{eq:first-est}) and (\ref{eq:second-est})
gives
\begin{equation}
\mathcal{M}(\phi_{n})\leq\mathcal{M^{*}}(\phi_{n})\eqdf\frac{1}{2\pi}\int_{0}^{2\pi}\|\L_{\tau,b+Re^{i\theta},\rho_{n}^{0}}\|_{\HS}^{2}d\theta+n^{-1}.\label{eq:MtoMstar}
\end{equation}
Combining (\ref{eq:NtoM}) and (\ref{eq:MtoMstar}) and taking expectations
gives
\begin{align*}
\E_{n}[\mathcal{N}(\phi_{n})] & \leq\log\left(\frac{R}{R'}\right)^{-1}\E_{n}[\mathcal{M^{*}}(\phi_{n})]\\
 & =\log\left(\frac{R}{R'}\right)^{-1}\frac{1}{2\pi}\int_{0}^{2\pi}\E_{n}[\|\L_{\tau,b+Re^{i\theta},\rho_{n}^{0}}\|_{\HS}^{2}]d\theta+n^{-1}.
\end{align*}
By Proposition \ref{prop:main-prob-estimate} we have $\E_{n}[\|\L_{\tau,b+Re^{i\theta},\rho_{n}^{0}}\|_{\HS}^{2}]\leq n^{-\epsilon}$
for all $\theta\in[0,2\pi].$ Hence 
\begin{equation}
\E_{n}[\mathcal{N}(\phi_{n})]\leq\log\left(\frac{R}{R'}\right)^{-1}\left(n^{-\epsilon}+n^{-1}\right)\label{eq:final-expectation-bound}
\end{equation}
for $n\geq n_{1}$. By Markov's inequality, the probability that $\zeta_{\tau,\rho_{n}^{0}}$
has at least one zero in $\Rect(\sigma_{0},H)$ is bounded by the
right hand side of (\ref{eq:final-expectation-bound}); since this
$\to0$ as $n\to\infty$, a.a.s. $\zeta_{\tau,\rho_{n}^{0}}$ has
no zeros in $\Rect(\sigma_{0},H)$. Hence by the previous arguments,
a.a.s.
\[
\mathcal{R}_{X_{n}}\bigcap\Rect(\sigma_{0},H)=\mathcal{R}_{X}\bigcap\Rect(\sigma_{0},H)
\]
and the multiplicities on both sides are the same. $\square$

\section{Proof of Theorem \ref{thm:high-frequency} about high frequency resonances\label{sec:Unitary-representations-and}}

This part is largely independent from the previous sections. Although
we use the technique of induced representations to keep track of resonances
in covers, we prove a spectral estimate on transfer operators twisted
by any unitary representation which implies Theorem \ref{thm:high-frequency}
via induced representations. We will prove the following completely
general fact. Let $\rho:\Gamma\rightarrow\mathcal{U}(V)$ be a unitary
representation of $\Gamma$ on a complex Hilbert\footnote{We do not assume that it is finite dimensional here.}
space $V$. Here $\mathcal{U}(V)$ is the set of unitary operators
on $V$. Recall that $I=\cup_{j=1}^{2r}I_{j}$. Let $C^{1}(I,V)$
denote the Banach space of $V$-valued functions, $C^{1}$ on $\overline{I}$,
endowed with the norm ($t\neq0$) 
\[
\Vert f\Vert_{(t),V}:=\Vert f\Vert_{\infty,V}+\frac{1}{\vert t\vert}\Vert f'\Vert_{\infty,V},
\]
where as usual 
\[
\Vert f\Vert_{\infty,V}=\sup_{x\in I}\Vert f(x)\Vert_{V},
\]
where $\Vert.\Vert_{V}$ is the Hilbert space norm on $V$. We recall
that the action of the ``basic'' transfer operator $\L_{s,\rho}$,
now on the function space $C^{1}(I,V)$, is given by 
\[
\L_{s,\rho}(F)(x)\eqdf\sum_{j\rightarrow i}(\gamma_{j}')^{s}(x)\rho(\gamma_{j}^{-1})F(\gamma_{j}x),\ \mathrm{if}\ x\in I_{i}.
\]
We will use the notation $\W_{N}^{j}\eqdf\left\{ {\bf a}\in\W_{N}\ :\ {\bf a}\rightarrow j\right\} $.
Given the previously defined notations and $F\in C^{1}(I,V)$, we
have for all $x\in I_{j}$ and $N\in\N$, 
\[
\L_{s,\rho}^{N}(F)(x)=\sum_{{\bf a}\in\mathscr{\W}_{N}^{j}}(\gamma_{{\bf a}}'(x))^{s}\rho(\gamma_{{\bf a}}^{-1})F(\gamma_{{\bf a}}(x)).
\]
We mention here that we could also alternatively use the ``refined''
transfer operator $\L_{\tau,s,\rho}$ here in place of $\L_{s,\rho}^{N}$,
but it wouldn't change the final result, nor it would make the size
of the gap explicit. We will need in this section some standard distortion
estimates. Some of them (bounded distortion) have already been used
in previous sections, but we recall them for the convenience of the
reader.
\begin{itemize}
\item \textit{(Uniform hyperbolicity)}. There exists $C>0$ and $0<\overline{\theta}<\theta<1$
such that for all $N$ and all $j$ such that ${\bf a}\in\mathscr{\W}_{N}^{j}$,
then for all $x\in I_{j}$ we have 
\[
C^{-1}\overline{\theta}^{N}\leq\vert\gamma'_{{\bf a}}(x)\vert\leq C\theta^{N}.
\]
\item \textit{(Bounded distortion).} There exists $M_{1}>0$ such that for
all $N,j$ and all ${\bf a}\in\mathscr{\W}_{N}^{j}$, 
\[
\sup_{I_{j}}\left\vert \frac{\gamma''_{{\bf a}}}{\gamma'_{{\bf a}}}\right\vert \leq M_{1}.
\]
\item \textit{(Bounded distortion for the third derivatives)}. There exists
$Q>0$ such that for all $n,j$ and all ${\bf a}\in\mathscr{\W}_{N}^{j}$,
\[
\sup_{I_{j}}\left\vert \frac{\gamma'''_{{\bf a}}}{\gamma'_{{\bf a}}}\right\vert \leq Q.
\]
\end{itemize}
Notice that the ``bounded distortion for the third derivatives''
follows directly from differentiating two times $\log(\gamma'_{{\bf a}})$,
and using bounded distortion and uniform hyperbolicity several times,
see for example \cite[\S 3]{Baladi} for a previous occurrence of
this condition in the literature. We now state the Ruelle-Perron-Frobenius
Theorem, which will be used below. The statement of this theorem in
the symbolic setting can be found in \cite[Thm. 2.2]{ParryPollicott}.
The version we use can be obtained via the work of Liverani \cite{Liverani}
as in \cite[Thm. 5.1]{NaudExpanding}.
\begin{thm}
\label{thm:RPF}Set $\L_{\sigma}=\L_{\sigma,\mathrm{Id}}$ where $\sigma$
is real and $\rho=\mathrm{Id}$ means the trivial one-dimensional
representation.
\end{thm}

\begin{enumerate}
\item \textit{The spectral radius of} $\L_{\sigma}$ \textit{on $C^{1}(I,\C)$
is $e^{P(\sigma)}$ which is a simple eigenvalue associated to a strictly
positive eigenfunction $h_{\sigma}>0$ in $C^{1}(I,\C)$.}
\item \textit{The operator $\L_{\sigma}$ on $C^{1}(I,\C)$ is quasi-compact
with essential spectral radius smaller than $\kappa(\sigma)e^{P(\sigma)}$
for some $\kappa(\sigma)<1$.}
\item \textit{There are no other eigenvalues on $\vert z\vert=e^{P(\sigma)}$.
Moreover, the spectral projector $\mathbb{P}_{\sigma}$ on $\{e^{P(\sigma)}\}$
is given by 
\[
\mathbb{P}_{\sigma}(f)=h_{\sigma}\int_{\Lambda(\Gamma)}fd\mu_{\sigma},
\]
where $\mu_{\sigma}$ is the unique probability measure on $\Lambda$
that satisfies $\L_{\sigma}^{*}(\mu_{\sigma})=e^{P(\sigma)}\mu_{\sigma}$,
and the eigenfunction $h_{\sigma}$ is normalized so that}
\[
\int h_{\sigma}d\mu_{\sigma}=1.
\]
\end{enumerate}
We continue with a basic a priori estimate.
\begin{lem}
\label{lem:LasotaYorke}Fix some $\sigma_{0}<\delta$, then there
exists $C_{0}>0,\ \rho<1$ such that for all $N$, all unitary representations
$(\rho,V)$ and all $s=\sigma+it$ with $\sigma\geq\sigma_{0}$, we
have 
\[
\Vert\left(\L_{s,\rho}^{N}(f)\right)'\Vert_{\infty,V}\leq C_{0}e^{NP(\sigma_{0})}\left\{ (1+\vert t\vert)\Vert f\Vert_{\infty,V}+\theta^{N}\Vert f'\Vert_{\infty,V}\right\} .
\]
\end{lem}

\begin{proof}
Differentiate the formula for $\L_{s,\rho}^{n}(f)$: since the representation
factor is locally constant, we don't need to differentiate it. Use
the bounded distortion property plus the uniform contraction, combined
with the pressure estimate in Lemma \ref{lem:Pressure-estimate}.
Uniformity with respect to $(\rho,V)$ follows from triangle inequality
plus the fact that for all $\gamma\in\Gamma$, we have $\Vert\rho(\gamma)\Vert_{V}=1$.
\end{proof}
\noindent The main fact of this section is the following. It is essentially
a vector-valued version of a result stated in \cite{JNS}. This type
of estimate is called a \textit{Dolgopyat estimate }by reference to
Dolgopyat's work on Anosov flows \cite{Dolg} where these type of
bounds appeared for the first time.
\begin{prop}
\label{Decay1} There exist $\varepsilon>0$, $T_{0}>0$ and $C_{1},\beta>0$
such that for all $N=N(t)=[C_{1}\log\vert t\vert]$ with $s=\sigma+it$
satisfying $\vert\sigma-\delta\vert\leq\varepsilon$ and $\vert t\vert\geq T_{0}$,
we have 
\[
\int_{\Lambda(\Gamma)}\Vert\L_{s,\rho}^{N}(f)\Vert_{V}^{2}d\mu_{\delta}\leq\frac{\Vert f\Vert_{(t),V}^{2}}{\vert t\vert^{\beta}}.
\]
All the constants here are uniform with respect to $\rho,V$.
\end{prop}

\noindent A particular case of this estimate was proved in \cite{OW,MOW}
for the case of congruence subgroups, where 
\[
\rho:\Gamma\rightarrow\mathcal{U}\left(L^{2}(\SL_{2}(\mathbb{F}_{p}))\right),
\]
is obtained after reduction mod $p$ via the regular representation
of $\SL_{2}(\mathbb{F}_{p})$. The proof was an adaptation of the
arguments of \cite{NaudExpanding}. We will present below a shorter,
more direct version of this estimate which allows to prove this generalization
without much effort.

Let us first briefly explain why this actually implies Theorem \ref{thm:high-frequency}.
We set $\rho=\mathrm{Ind}_{\widetilde{\Gamma}}^{\Gamma}$, where $\widetilde{\Gamma}$
is an arbitrary, finite index subgroup of $\Gamma$, and $\mathrm{Ind}_{\widetilde{\Gamma}}^{\Gamma}$
is the induced representation to $\Gamma$ of the trivial representation
of $\widetilde{\Gamma}$. We work by contradiction. Assume that $Z_{\widetilde{\Gamma}}(s)=0$,
then according to the induction formula of Venkov-Zograf \cite{VZ,FP},
we have for $s=\sigma+it$, 
\[
\L_{s,\rho}(F_{s})=F_{s},
\]
for some $F_{s}\not\equiv0\in C^{1}(I,V)$. We can definitely normalize
$F_{s}$ so that $\Vert F_{s}\Vert_{(t),V}=1$. Write $N=N_{1}+N(t)$,
where $N(t)$ is given by Proposition \ref{Decay1}. Take $\sigma_{0}\leq\sigma\leq\delta$.
Using the triangle inequality for $\Vert.\Vert_{V}$ and unitarity
of $\rho$, we have (by Cauchy-Schwarz) and the pressure estimate
(Lemma \ref{lem:Pressure-estimate}), 
\[
\Vert F_{s}\Vert_{\infty,V}\leq C_{0}e^{\frac{N_{1}}{2}P(2\sigma_{0}-\delta)}\left(\L_{\delta}^{N_{1}}\left(\Vert\L_{s,\rho}^{N(t)}(F_{s})\Vert_{V}^{2}\right)\right)^{1/2}.
\]
We need to estimate the $C^{1}$-norm of $x\mapsto\Vert\L_{s,\rho}^{N(t)}(F_{s})\Vert_{V}^{2}(x)$
on $I$. Since we work with a Hilbert norm, the square of the norm
is differentiable and we can compute 
\[
\frac{d}{dx}\Vert\L_{s,\rho}^{N(t)}(F_{s})\Vert_{V}^{2}=2\Re\left(\langle(\L_{s,\rho}^{N(t)}(F_{s}))',\L_{s,\rho}^{N(t)}(F_{s})\rangle_{V}\right),
\]
and use the $V$-valued Lasota-Yorke estimate from Lemma \ref{lem:LasotaYorke}
and Cauchy-Schwarz to obtain 
\[
\Vert\Vert\L_{s,\rho}^{N(t)}(F_{s})\Vert_{V}^{2}\Vert_{C^{1}(I)}\leq Ce^{2N(t)P(\sigma_{0})}(1+\vert t\vert).
\]
Using the Ruelle-Perron-Frobenius Theorem (Theorem \ref{thm:RPF}),
and the fact that $P(\delta)=0$, we get 
\[
\Vert F_{s}\Vert_{\infty,V}^{2}\leq Ce^{N_{1}P(2\sigma_{0}-\delta)}\left(\int_{\Lambda(\Gamma)}\Vert\L_{s,\rho}^{N(t)}(F_{s})\Vert_{V}^{2}d\mu_{\delta}+\kappa_{\delta}^{N_{1}}e^{2N(t)P(\sigma_{0})}(1+\vert t\vert)\right),
\]
with $\kappa_{\delta}<1$. Assuming that $\sigma_{0}\geq\delta-\epsilon$
and $\vert t\vert\geq T_{0}$, we can apply Proposition \ref{Decay1}
and set $N_{1}=N_{1}(t)=[C_{2}\log\vert t\vert]$ to get 
\[
\Vert F_{s}\Vert_{\infty,V}^{2}\leq C\left(\vert t\vert^{C_{2}P(2\sigma_{0}-\delta)-\beta}+\vert t\vert^{-C_{2}\vert\log\kappa\vert+2C_{1}P(\sigma_{0})+1}\right).
\]
We then take $C_{2}$ large enough and fix $\sigma_{0}$ close enough
to $\delta$ so that $C_{2}P(2\sigma_{0}-\delta)-\beta<0$ and we
get 
\[
\Vert F_{s}\Vert_{\infty,V}^{2}\leq C\vert t\vert^{-\widetilde{\beta}},
\]
for some $\widetilde{\beta}>0$. The same calculation can be performed
to obtain similarly 
\[
\Vert F_{s}'\Vert_{\infty,V}^{2}\leq C\vert t\vert^{-\widetilde{\beta}+2},
\]
and we reach a contradiction for all $\vert t\vert$ large since $1=\Vert F_{s}\Vert_{(t),V}\leq C'\vert t\vert^{-\widetilde{\beta}/2}.$
Once again, all the constants are uniform with respect to $(\rho,V)$.

\bigskip{}
The proof of the key Proposition \ref{Decay1} will rest on the following
result of Bourgain-Dyatlov \cite{BDFourier}.
\begin{thm}
\label{Decay2} There exist constants $\beta_{1},\beta_{2}>0$ such
that the following holds. Given $g\in C^{1}(I)$ and $\Phi\in C^{2}(I)$,
consider the integral 
\[
\mathcal{I}(\xi)\eqdf\int_{\Lambda(\Gamma)}e^{-i\xi\Phi(x)}g(x)d\mu_{\delta}(x).
\]
If we have 
\[
\inf_{\Lambda(\Gamma)}\vert\Phi'\vert\geq\vert\xi\vert^{-\beta_{1}},
\]
and $\Vert\Phi\Vert_{C^{2}}\leq M$, then for all $\vert\xi\vert\geq1$,
we have 
\[
\vert\mathcal{I}(\xi)\vert\leq C_{M}\vert\xi\vert^{-\beta_{2}}\Vert g\Vert_{C^{1}},
\]
where $C_{M}>0$ does not depend on $\xi,g$.
\end{thm}

For comments on this version of the Bourgain-Dyatlov decay estimate,
see \cite{JNS}. Let us just mention that $\mu_{\delta}$, up to a
smooth density, is the Patterson-Sullivan measure, see \cite{JNS}.
To be able to use this estimate, we will use the following fact from
\cite{JNS}, which is referred there as the ``uniform-non-integrability
property'' (UNI), see Proposition 4.10.
\begin{prop}
\label{UNI1} (UNI) For all ${\bf a},{\bf b}\in\mathscr{\W}_{N}^{j}$
set 
\[
\mathscr{D}({\bf a},{\bf b}):=\inf_{x\in I_{j}}\left\vert \frac{\gamma''_{{\bf a}}(x)}{\gamma'_{{\bf a}}(x)}-\frac{\gamma''_{{\bf b}}(x)}{\gamma'_{{\bf b}}(x)}\right\vert .
\]
There exist constants $M>0$ and $\eta_{0}>0$ such that for all $n$
and all $\epsilon=e^{-\eta N}$ with $0<\eta<\eta_{0}$, we have for
all ${\bf a}\in\mathscr{\mathcal{W}}_{N}^{j}$, 
\[
\sum_{{\bf b}\in\mathscr{\W}_{N}^{j},\ \mathscr{D}({\bf a},{\bf b})<\epsilon}\Vert\gamma'_{{\bf b}}\Vert_{I_{j},\infty}^{\delta}\leq M\epsilon^{\delta}.
\]
\end{prop}

For a proof of that fact, see \cite[\S 4]{JNS}. We are now ready
to conclude this section by the proof of Proposition \ref{Decay1}.
Pick $f\in C^{1}(I,V)$. We set $s=\sigma+it$ and we assume that
$\sigma$ is close to $\delta$. Let us write 
\[
\mathcal{S}_{\sigma,N}(t):=\int_{\Lambda(\Gamma)}\Vert\L_{s,\rho}^{N}(f)\Vert_{V}^{2}d\mu_{\delta}=\sum_{j=1}^{2r}\sum_{{\bf a},{\bf b}\in\mathscr{\W}_{N}^{j}}\int_{\Lambda(\Gamma)}e^{it\Phi_{{\bf a},{\bf b}}(x)}g_{{\bf a},{\bf b}}^{(j)}(x)d\mu_{\delta}(x),
\]
with 
\[
\Phi_{{\bf a},{\bf b}}(x)=\log\gamma_{{\bf a}}'(x)-\log\gamma_{{\bf b}}'(x),
\]
and 
\[
g_{{\bf a},{\bf b}}^{(j)}(x)=\begin{cases}
\left(\gamma'_{{\bf a}}(x)\right)^{\sigma}\left(\gamma'_{{\bf b}}(x)\right)^{\sigma}\langle\rho(\gamma_{{\bf a}}^{-1})f\circ\gamma_{{\bf a}}(x),\rho(\gamma_{{\bf b}}^{-1})f\circ\gamma_{{\bf b}}(x)\rangle_{V}\ \mathrm{if\ }x\in I_{j},\\
0\ \mathrm{otherwise.}
\end{cases}
\]
 Notice that $g_{{\bf a},{\bf b}}^{(j)}$ is indeed a $C^{1}$ function
on a neighborhood of $\Lambda(\Gamma)$. By using the bounded distortion
property and Cauchy-Schwarz we have easily:
\begin{equation}
\sup_{I_{j}}\vert g_{{\bf a},{\bf b}}^{(j)}\vert\leq C_{1}\sup_{I_{j}}\vert\gamma'_{{\bf a}}\vert^{\sigma}\sup_{I_{j}}\vert\gamma'_{{\bf b}}\vert^{\sigma}\Vert f\Vert_{(t),V}^{2}.\label{est1}
\end{equation}
Differentiating inside the inner product $\langle.,.\rangle_{V}$
and using the bounded distortion plus the uniform contraction (with
Cauchy-Schwarz again) gives also
\begin{equation}
\sup_{I_{j}}\left\vert \frac{d}{dx}g_{{\bf a},{\bf b}}^{(j)}\right\vert \leq C_{2}\sup_{I_{j}}\vert\gamma'_{{\bf a}}\vert^{\sigma}\sup_{I_{j}}\vert\gamma'_{{\bf b}}\vert^{\sigma}(1+\vert t\vert\theta^{N})\Vert f\Vert_{(t),V}^{2}.\label{est2}
\end{equation}
Both estimates (\ref{est1}) and (\ref{est2}) can be combined to
yield
\begin{equation}
\Vert g_{{\bf a},{\bf b}}^{(j)}\Vert_{C^{1}}\leq\widetilde{C_{2}}\sup_{I_{j}}\vert\gamma'_{{\bf a}}\vert^{\sigma}\sup_{I_{j}}\vert\gamma'_{{\bf b}}\vert^{\sigma}(2+\vert t\vert\theta^{N})\Vert f\Vert_{(t),V}^{2}.\label{est3}
\end{equation}
We also observe that $\inf_{x\in I_{j}}\vert\Phi_{{\bf a},{\bf b}}'(x)\vert=\mathscr{D}({\bf a},{\bf b})$,
and that by using the bounded distortion for the second and third
derivatives we have for some uniform $C_{3}>0$, 
\[
\Vert\Phi_{{\bf a},{\bf b}}\Vert_{C^{2}}\leq C_{3}.
\]
The plan is now to split $\mathcal{S}_{\sigma,N}(t)$ as 
\[
\mathcal{S}_{\sigma,N}(t)=\mathcal{S}_{\sigma,N}^{(1)}(t)+\mathcal{S}_{\sigma,N}^{(2)}(t),
\]
with the ``near-diagonal'' sum
\[
\mathcal{S}_{\sigma,N}^{(1)}(t):=\sum_{j=1}^{2r}\sum_{\mathscr{D}({\bf a},{\bf b})\leq\epsilon}\int_{\Lambda(\Gamma)}e^{it\Phi_{{\bf a},{\bf b}}(x)}g_{{\bf a},{\bf b}}^{(j)}(x)d\mu_{\delta}(x),
\]
and the ``off-diagonal`` sum
\[
\mathcal{S}_{\sigma,N}^{(2)}(t):=\sum_{j=1}^{2r}\sum_{\mathscr{D}({\bf a},{\bf b})>\epsilon}\int_{\Lambda(\Gamma)}e^{it\Phi_{{\bf a},{\bf b}}(x)}g_{{\bf a},{\bf b}}^{(j)}(x)d\mu_{\delta}(x),
\]
with $\epsilon>0$. We now assume that $\sigma_{0}\leq\sigma\leq\delta$
and $N=[\kappa\log\vert t\vert]$, with $\epsilon=e^{-\eta N}$with
$0<\eta<\eta_{0}$. We \textit{fix} $\kappa$ large enough so that
$\vert t\vert\theta^{N}$stays uniformly bounded as $\vert t\vert\rightarrow\infty$,
and pick $\eta>0$ small enough such that $\epsilon=e^{-\eta[\kappa\log\vert t\vert]}>\vert t\vert^{^{-\beta_{1}}}$
, so that we can apply Theorem \ref{Decay2}. Combining estimate (\ref{est3})
with the pressure bound from Lemma \ref{lem:Pressure-estimate}, we
get
\[
\vert\mathcal{S}_{\sigma,N}^{(2)}(t)\vert\leq C\frac{\Vert f\Vert_{(t),V}^{2}}{\vert t\vert^{\beta_{2}}}e^{2NP(\sigma_{0})}.
\]
On the other hand we have
\[
\vert\mathcal{S}_{\sigma,N}^{(1)}(t)\vert\leq C\Vert f\Vert_{(t),V}^{2}\sum_{j}\sum_{{\bf a}\in\mathcal{W}_{N}^{j}}\sup_{I_{j}}\vert\gamma'_{{\bf {\bf a}}}\vert^{\sigma}\sum_{{\bf b}:\mathscr{D}({\bf a},{\bf b})<\epsilon}\sup_{I_{j}}\vert\gamma'_{{\bf b}}\vert^{\sigma},
\]
which by using Proposition \ref{UNI1} and the pressure estimate combined
with the uniform hyperbolicity (the lower bound) gives
\[
\vert\mathcal{S}_{\sigma,N}^{(1)}(t)\vert\leq C\Vert f\Vert_{(t),V}^{2}e^{NP(\sigma_{0})}\overline{\theta}^{N(\sigma_{0}-\delta)}\epsilon^{\delta}.
\]
because $P(\sigma_{0})\rightarrow0$ as $\sigma_{0}\rightarrow\delta$,
we can definitely pick $\sigma_{0}<\delta$ so that for all $\vert t\vert\geq1$,
we have
\[
\mathcal{\vert S}_{\sigma,N}(t)\vert\leq\vert\mathcal{S}_{\sigma,N}^{(1)}(t)\vert+\vert\mathcal{S}_{\sigma,N}^{(2)}(t)\vert\leq\widetilde{C}\Vert f\Vert_{(t),V}^{2}\vert t\vert^{-\widetilde{\beta}},
\]
for some uniform $\widetilde{C}>0$ and $\widetilde{\beta}>0$. This
ends the proof.

\bibliographystyle{alpha}
\bibliography{database}

\begin{thebibliography}{{Dya}19}

\bibitem[Alo86]{Alon}
N.~Alon.
\newblock Eigenvalues and expanders.
\newblock {\em Combinatorica}, 6(2):83--96, 1986.
\newblock Theory of computing (Singer Island, Fla., 1984).

\bibitem[BD17]{BDFourier}
J.~Bourgain and S.~Dyatlov.
\newblock Fourier dimension and spectral gaps for hyperbolic surfaces.
\newblock {\em Geom. Funct. Anal.}, 27(4):744--771, 2017.

\bibitem[BD18]{BDgap}
J.~Bourgain and S.~Dyatlov.
\newblock Spectral gaps without the pressure condition.
\newblock {\em Ann. of Math. (2)}, 187(3):825--867, 2018.

\bibitem[BGS11]{BGS2}
J.~Bourgain, A.~Gamburd, and P.~Sarnak.
\newblock Generalization of {S}elberg's {$\frac{3}{16}$} theorem and affine
  sieve.
\newblock {\em Acta Math.}, 207(2):255--290, 2011.

\bibitem[BM04]{BrooksMakover}
R.~Brooks and E.~Makover.
\newblock Random construction of {R}iemann surfaces.
\newblock {\em J. Differential Geom.}, 68(1):121--157, 2004.

\bibitem[BMM17]{BMM}
W.~Ballmann, H.~Matthiesen, and S.~Mondal.
\newblock Small eigenvalues of surfaces of finite type.
\newblock {\em Compos. Math.}, 153(8):1747--1768, 2017.

\bibitem[Bol88]{Bollobas}
B.~Bollob\'{a}s.
\newblock The isoperimetric number of random regular graphs.
\newblock {\em European J. Combin.}, 9(3):241--244, 1988.

\bibitem[Bor16]{Borthwick}
D.~Borthwick.
\newblock {\em Spectral theory of infinite-area hyperbolic surfaces}, volume
  318 of {\em Progress in Mathematics}.
\newblock Birkh\"{a}user/Springer, [Cham], second edition, 2016.

\bibitem[Bow79]{Bowen}
R.~Bowen.
\newblock Hausdorff dimension of quasicircles.
\newblock {\em Inst. Hautes \'{E}tudes Sci. Publ. Math.}, (50):11--25, 1979.

\bibitem[Bro86]{Brooks}
R.~Brooks.
\newblock The spectral geometry of a tower of coverings.
\newblock {\em J. Differential Geom.}, 23(1):97--107, 1986.

\bibitem[BS87]{BroderShamir}
A.~Broder and E.~Shamir.
\newblock On the second eigenvalue of random regular graphs.
\newblock In {\em The 28th Annual Symposium on Foundations of Computer
  Science}, pages 286--294, 1987.

\bibitem[Bur88]{Burger}
M.~Burger.
\newblock Spectre du laplacien, graphes et topologie de {F}ell.
\newblock {\em Comment. Math. Helv.}, 63(2):226--252, 1988.

\bibitem[But98]{Button}
J.~Button.
\newblock All {F}uchsian {S}chottky groups are classical {S}chottky groups.
\newblock In {\em The {E}pstein birthday schrift}, volume~1 of {\em Geom.
  Topol. Monogr.}, pages 117--125. Geom. Topol. Publ., Coventry, 1998.

\bibitem[BV05]{Baladi}
V.~Baladi and B.~Vall\'{e}e.
\newblock Euclidean algorithms are {G}aussian.
\newblock {\em J. Number Theory}, 110(2):331--386, 2005.

\bibitem[DJ18]{DJ1}
S.~Dyatlov and L.~Jin.
\newblock Dolgopyat's method and the fractal uncertainty principle.
\newblock {\em Anal. PDE}, 11(6):1457--1485, 2018.

\bibitem[Dol98]{Dolg}
D.~Dolgopyat.
\newblock On decay of correlations in {A}nosov flows.
\newblock {\em Ann. of Math. (2)}, 147(2):357--390, 1998.

\bibitem[{Dya}19]{DyatlovFUP}
S.~{Dyatlov}.
\newblock {An introduction to fractal uncertainty principle}.
\newblock {\em arXiv:1903.02599}, 2019.

\bibitem[DZ17]{DZFractal}
S.~{Dyatlov} and M.~{Zworski}.
\newblock {Fractal uncertainty for transfer operators}.
\newblock {\em arXiv:1710.05430}, page arXiv:1710.05430, Oct 2017.

\bibitem[FP17]{FP}
K.~{Fedosova} and A.~{Pohl}.
\newblock {Meromorphic continuation of Selberg zeta functions with twists
  having non-expanding cusp monodromy}.
\newblock {\em arXiv:1709.00760}, page arXiv:1709.00760, Sep 2017.

\bibitem[Fri08]{Friedman}
J.~Friedman.
\newblock A proof of {A}lon's second eigenvalue conjecture and related
  problems.
\newblock {\em Mem. Amer. Math. Soc.}, 195(910):viii+100, 2008.

\bibitem[Gam02]{Gamburd1}
A.~Gamburd.
\newblock On the spectral gap for infinite index ``congruence'' subgroups of
  {${\rm SL}_2(\bold Z)$}.
\newblock {\em Israel J. Math.}, 127:157--200, 2002.

\bibitem[Gam06]{GamburdBelyi}
A.~Gamburd.
\newblock Poisson-{D}irichlet distribution for random {B}elyi surfaces.
\newblock {\em Ann. Probab.}, 34(5):1827--1848, 2006.

\bibitem[GLZ04]{GLZ}
L.~Guillop\'{e}, K.~K. Lin, and M.~Zworski.
\newblock The {S}elberg zeta function for convex co-compact {S}chottky groups.
\newblock {\em Comm. Math. Phys.}, 245(1):149--176, 2004.

\bibitem[GN09]{GN1}
C.~Guillarmou and F.~Naud.
\newblock Wave decay on convex co-compact hyperbolic manifolds.
\newblock {\em Comm. Math. Phys.}, 287(2):489--511, 2009.

\bibitem[Gui92]{Guillope}
L.~Guillop\'{e}.
\newblock Fonctions z\^{e}ta de {S}elberg et surfaces de g\'{e}om\'{e}trie
  finie.
\newblock In {\em Zeta functions in geometry ({T}okyo, 1990)}, volume~21 of
  {\em Adv. Stud. Pure Math.}, pages 33--70. Kinokuniya, Tokyo, 1992.

\bibitem[GZ95]{GZ1}
L.~Guillop\'{e} and M.~Zworski.
\newblock Upper bounds on the number of resonances for non-compact {R}iemann
  surfaces.
\newblock {\em J. Funct. Anal.}, 129(2):364--389, 1995.

\bibitem[JN16]{JNIJM}
D.~Jakobson and F.~Naud.
\newblock Resonances and density bounds for convex co-compact congruence
  subgroups of {$SL_2(\Bbb{Z})$}.
\newblock {\em Israel J. Math.}, 213(1):443--473, 2016.

\bibitem[JNS19]{JNS}
D.~Jakobson, F.~Naud, and L.~Soares.
\newblock Large covers and sharp resonances of hyperbolic surfaces.
\newblock {\em To appear, Ann. Institut Fourier.}, 2019.

\bibitem[JZ17]{JZ1}
L.~Jin and R.~Zhang.
\newblock Fractal uncertainty principle with explicit exponent.
\newblock {\em Preprint}, 2017.

\bibitem[Liv95]{Liverani}
C.~Liverani.
\newblock Decay of correlations.
\newblock {\em Ann. of Math. (2)}, 142(2):239--301, 1995.

\bibitem[LP81]{LP}
P.~D. Lax and R.~S. Phillips.
\newblock The asymptotic distribution of lattice points in {E}uclidean and
  non-{E}uclidean spaces.
\newblock In {\em Functional analysis and approximation ({O}berwolfach, 1980)},
  volume~60 of {\em Internat. Ser. Numer. Math.}, pages 373--383.
  Birkh\"{a}user, Basel-Boston, Mass., 1981.

\bibitem[LPS88]{LPS}
A.~Lubotzky, R.~Phillips, and P.~Sarnak.
\newblock Ramanujan graphs.
\newblock {\em Combinatorica}, 8(3):261--277, 1988.

\bibitem[Mag15]{Magee}
M.~Magee.
\newblock Quantitative spectral gap for thin groups of hyperbolic isometries.
\newblock {\em J. Eur. Math. Soc. (JEMS)}, 17(1):151--187, 2015.

\bibitem[MM87]{MazzeoMelrose}
R.~R. Mazzeo and R.~B. Melrose.
\newblock Meromorphic extension of the resolvent on complete spaces with
  asymptotically constant negative curvature.
\newblock {\em J. Funct. Anal.}, 75(2):260--310, 1987.

\bibitem[MOW17]{MOW}
M.~Magee, H.~Oh, and D.~Winter.
\newblock Uniform congruence counting for {S}chottky semigroups in
  $\mathrm{SL}_2(\mathbf{Z})$), with appendix by {J}. {B}ourgain, {A}.
  {K}ontorovich, and {M}. {M}agee.
\newblock {\em Journal f\"{u}r die reine und angewandte Mathematik (Crelles
  Journal)}, 01 2017.

\bibitem[Nau05a]{NaudExpanding}
F.~Naud.
\newblock Expanding maps on {C}antor sets and analytic continuation of zeta
  functions.
\newblock {\em Ann. Sci. \'{E}cole Norm. Sup. (4)}, 38(1):116--153, 2005.

\bibitem[Nau05b]{Naudasymptotics}
F.~Naud.
\newblock Precise asymptotics of the length spectrum for finite-geometry
  {R}iemann surfaces.
\newblock {\em Int. Math. Res. Not.}, (5):299--310, 2005.

\bibitem[Nau14]{Naud}
F.~Naud.
\newblock Density and location of resonances for convex co-compact hyperbolic
  surfaces.
\newblock {\em Invent. Math.}, 195(3):723--750, 2014.

\bibitem[Nil91]{Nilli}
A.~Nilli.
\newblock On the second eigenvalue of a graph.
\newblock {\em Discrete Math.}, 91(2):207--210, 1991.

\bibitem[OW16]{OW}
H.~Oh and D.~Winter.
\newblock Uniform exponential mixing and resonance free regions for convex
  cocompact congruence subgroups of {$\operatorname{SL}_2(\Bbb{Z})$}.
\newblock {\em J. Amer. Math. Soc.}, 29(4):1069--1115, 2016.

\bibitem[Pat76]{Patterson}
S.~J. Patterson.
\newblock The limit set of a {F}uchsian group.
\newblock {\em Acta Math.}, 136(3-4):241--273, 1976.

\bibitem[PP90]{ParryPollicott}
W.~Parry and M.~Pollicott.
\newblock Zeta functions and the periodic orbit structure of hyperbolic
  dynamics.
\newblock {\em Ast\'{e}risque}, (187-188):268, 1990.

\bibitem[PP01]{PattersonPerry}
S.~J. Patterson and Peter~A. Perry.
\newblock The divisor of {S}elberg's zeta function for {K}leinian groups.
\newblock {\em Duke Math. J.}, 106(2):321--390, 2001.
\newblock Appendix A by Charles Epstein.

\bibitem[PP15]{PP15}
D.~Puder and O.~Parzanchevski.
\newblock Measure preserving words are primitive.
\newblock {\em Journal of the American Mathematical Society}, 28(1):63--97,
  2015.

\bibitem[Pud15]{PUDER}
D.~Puder.
\newblock Expansion of random graphs: new proofs, new results.
\newblock {\em Invent. Math.}, 201(3):845--908, 2015.

\bibitem[Sel65]{Selberg}
A.~Selberg.
\newblock On the estimation of {F}ourier coefficients of modular forms.
\newblock In {\em Proc. {S}ympos. {P}ure {M}ath., {V}ol. {VIII}}, pages 1--15.
  Amer. Math. Soc., Providence, R.I., 1965.

\bibitem[VZ82]{VZ}
A.~B. Venkov and P.~G. Zograf.
\newblock Analogues of {A}rtin's factorization formulas in the spectral theory
  of automorphic functions associated with induced representations of
  {F}uchsian groups.
\newblock {\em Izv. Akad. Nauk SSSR Ser. Mat.}, 46(6):1150--1158, 1343, 1982.

\bibitem[Zwo17]{Zworski_survey}
M.~Zworski.
\newblock Mathematical study of scattering resonances.
\newblock {\em Bull. Math. Sci.}, 7(1):1--85, 2017.

\end{thebibliography}

\newpage{}

\noindent Michael Magee, \\
Department of Mathematical Sciences,\\
Durham University, \\
Lower Mountjoy, DH1 3LE Durham,\\
United Kingdom

\noindent \texttt{michael.r.magee@durham.ac.uk}\\

\noindent Frédéric Naud, \\
Institut de Mathématiques de Jussieu, \\
Sorbonne Université, \\
4 place de Jussieu, \\
75252 Paris Cedex 05, France\\
\texttt{frederic.naud@imj-prg.fr}
\end{document}